\documentclass[11pt]{amsart}
\usepackage{amsmath,amssymb,amsfonts,mathtools}
\allowdisplaybreaks[1]
\usepackage[margin=1.05in]{geometry}
\usepackage[colorlinks=true,linkcolor=blue,citecolor=blue,urlcolor=blue]{hyperref}
\usepackage[capitalize]{cleveref}

\allowdisplaybreaks

\newtheorem{theorem}{Theorem}[section]
\newtheorem{proposition}[theorem]{Proposition}
\newtheorem{lemma}[theorem]{Lemma}
\newtheorem{corollary}[theorem]{Corollary}
\theoremstyle{definition}

\theoremstyle{remark}
\newtheorem{remark}[theorem]{Remark}

\newcommand{\R}{\mathbb R}
\newcommand{\bfK}{\mathbf K}
\newcommand{\one}{\mathbf 1}
\newcommand{\supp}{\operatorname{supp}}
\newcommand{\dist}{\operatorname{dist}}
\newcommand{\diam}{\operatorname{diam}}
\newcommand{\ess}{\mathrm{ess}}
\newcommand{\HS}{\mathrm{HS}}

\newcommand{\ip}[2]{\langle #1,#2\rangle}

\title[Compact two-weight Riesz transforms on curves]
{Compactness and essential norms for two-weight\\
fractional Riesz transforms on curves}
\author{Ji Li}
\address{School of Mathematical and Physical Sciences, Macquarie University,
NSW 2109, Australia}
\email{ji.li@mq.edu.au}
\author{Chun-Yen Shen}
\address{Department of Mathematics, National Taiwan University,
Taipei 106, Taiwan}
\email{cyshen@math.ntu.edu.tw}
\author{Chaojie Wen}
\address{School of Mathematical and Physical Sciences, Macquarie University,
NSW 2109, Australia}
\email{chaojie.wen@students.mq.edu.au}
\author{Brett D. Wick}
\address{Department of Mathematics, Washington University in St. Louis,
One Brookings Drive, St. Louis, Missouri 63130-4899, USA}
\email{wick@math.wustl.edu}
\author{Liangchuan Wu}
\address{School of Mathematical Sciences, Anhui University,
Hefei 230601, China}
\email{wuliangchuan@ahu.edu.cn}
\date{}

\begin{document}
\raggedbottom
\begin{abstract}
We prove a compact two-weight $T1$ theorem for the $\alpha$-fractional Riesz vector, $0\leq\alpha<n$, when one measure is supported on a compact regular $C^{1,\mathrm{Dini}}$ curve $\Gamma$. The measures may have common point masses. Under the global boundedness assumptions, we give a two-sided estimate for the essential norm in terms of Muckenhoupt and testing constants computed after restricting both measures to small cubes. Compactness is equivalent to the uniform vanishing of these local bounds as the cube side length tends to zero. The smooth truncations have a unique weak operator limit and converge in operator norm in the compact case.

For arclength measure and distance power weights, we determine the compactness ranges for Green gradients on bounded tubes and for Newtonian gradients and Biot--Savart fields on the whole space. The whole-space operators in these ranges are compact but are not Hilbert--Schmidt. When $\alpha=n-1$ and both measures have positive continuous arclength densities $a,b$ on the same closed curve, we obtain the exact essential norm $\pi c_{n-1,n}\max_\Gamma\sqrt{ab}$.
\end{abstract}
\maketitle

\section{Introduction}

Two-weight inequalities ask when a singular integral maps $L^2(\sigma)$ boundedly into $L^2(\omega)$ for a pair
of positive Borel measures on which no structural assumptions are placed. The measures need not be doubling or
absolutely continuous, they may carry point masses, and they need not be related to each other in any way. In this
generality the tools of classical weighted theory are unavailable, and the aim is to characterize boundedness by a
finite list of conditions that test the operator and the measures on cubes. For the Hilbert transform on the line,
Lacey, Sawyer, Shen and Uriarte-Tuero \cite{HilbertI} and Lacey \cite{HilbertII} characterized two-weight $L^2$
boundedness by a Poisson $A_2$ condition and forward and dual testing for measures without common atoms, and
Hyt\"onen \cite{HilbertGeneral} obtained a characterization for general positive Radon measures, including common
atoms. See also the closely related results in \cite{SawyerMaximal}, \cite{SawyerFractional}, \cite{FailureEnergy},
\cite{EnergyCommon}, \cite{FGW} and \cite{GoodLambda}. In higher dimensions, the companion paper \cite[Theorem~1.3]{BoundedDini}
characterizes boundedness of the fractional Riesz vector when one measure is supported on a compact regular
$C^{1,\mathrm{Dini}}$ curve, by two one-tailed Muckenhoupt conditions, two punctured conditions, and forward and dual
testing over rotated cubes. It extends the $C^{1,\delta}$-curve theorem of Sawyer, Shen and Uriarte-Tuero
\cite[Theorem~9]{CurvePaper}.

Once boundedness is understood, the next question is compactness, and its quantitative form is the essential norm,
the distance from the operator to the compact operators. The essential norm measures exactly the part of the operator
that no finite-rank approximation can remove. In the present setting it also governs two concrete questions.
First, the operator is defined only as a limit of smooth truncations, and one wants to know when that limit is
attained in operator norm rather than merely weakly. Second, when the curve carries a charge or a current and the
operator is the gradient of the resulting potential or field, compactness is precisely the statement that weakly
convergent source densities produce strongly convergent fields with convergent energies.

This paper answers these questions for the operators and measures of \cite{BoundedDini}. We give a two-sided
estimate for the essential norm in terms of the boundedness characterization computed after restricting both
measures to small cubes, characterize compactness, and prove norm convergence of the truncations in the compact
case. For positive continuous densities on a closed curve we obtain the essential norm exactly, and for distance
power weights we determine the compactness ranges of Green gradients, Newtonian gradients and Biot--Savart fields.

More precisely, for $0\leq\alpha<n$ and $d=n-\alpha$ the fractional Riesz vector has kernel
$$
 \bfK^{\alpha,n}(x,y)=c_{\alpha,n}\frac{y-x}{|y-x|^{d+1}},
 \qquad x\ne y,\qquad c_{\alpha,n}>0.
$$
We study compactness from $L^2(\sigma)$ to $L^2(\omega;\mathbb C^n)$ when one measure is supported on a compact
regular $C^{1,\mathrm{Dini}}$ curve $\Gamma$. The measures are locally finite and positive, and may have common atoms.
We first assume $\supp\sigma\subset\Gamma$, and the case $\supp\omega\subset\Gamma$ follows by componentwise duality.

\subsection{Compactness in the two-weight setting}

Compact $T1$ theorems are known in other settings. Lacey, Sawyer, Shen, Uriarte-Tuero and Wick
\cite[Theorem~1.9]{Cauchy} give conditions for compactness of the Cauchy transform from the circle to the disc using
vanishing Poisson and testing conditions.
Villarroya \cite{Villarroya2015} proved a compact $T1$ theorem on unweighted $L^p(\R)$, $1<p<\infty$, and his
compact local $T1$ theorem \cite[Theorem~4.2]{Villarroya} applies to a single measure of positive power growth.
Mitkovski and Stockdale \cite{MitkovskiStockdale} characterize compact Calder\'on--Zygmund operators on unweighted
$L^2(\R^n)$ by weak compactness and $T1,T^*1\in\mathrm{CMO}$. In each of these results the measure is either
Lebesgue measure or a single measure with a growth condition, and the operator is given in advance.
Here the two measures are arbitrary, the curve-supported measure may be singular or atomic, the two measures may
share point masses, and the other measure need not have compact support.

Three features of this setting distinguish it from the results above and shape both the statement and the proof.

First, there is no canonical operator to begin with. When $\sigma$ is singular or atomic there is no principal
value to which the kernel can be assigned, and the operator exists only as a weak operator limit of smooth
truncations along a subsequence. That this limit is independent of the subsequence is itself a result,
\cref{prop:unique}, and we prove it without any compactness assumption: two weak limits agree off the diagonal, so
their difference is a multiplication operator, and a symmetric change of density shows that the multiplier is real
and skew-adjoint, hence zero.

Second, common atoms are a genuine obstruction. A point mass carried by both measures contributes nothing to any
truncated operator, since every cutoff vanishes on the diagonal. The same vanishing prevents the truncations from
converging in norm to any operator with a nonzero atomic diagonal, even a compact one, and it breaks the standard
lower bound for the essential norm, because projections onto shrinking cubes around a common atom do not tend
strongly to zero.

Third, and most important for the form of the result, the usual vanishing testing conditions are not the right
criterion. In the unweighted theory, the $\mathrm{CMO}$ conditions require the testing quantities to vanish at
small scales, at large scales and at infinity. In the two-weight setting, the large-scale conditions fail for
compact operators. For two point masses at distinct points the operator has rank one, yet the normalized testing quantities over
every cube containing both points are bounded below, independently of the side length, see
Section~\ref{sec:examples}. A criterion built from vanishing testing conditions over large cubes would therefore exclude
compact operators.

The resolution is to localize differently. Fix an axis-parallel cube $D$ of small side length, restrict both
measures to $D$, and compute the full boundedness characterization of \cite{BoundedDini} for the restricted pair
$(\one_D\sigma,\one_D\omega)$, recomputing the common atoms and the punctured masses after restriction. The
essential norm is comparable to the supremum of these local characterizations as the side length of $D$ tends to
zero. No condition at large scales or at infinity is needed: because $\sigma$ has compact support, boundedness
alone forces the kernel away from the diagonal to be Hilbert--Schmidt, even when $\omega$ has unbounded support,
and every large-scale contribution is compact. Common atoms are handled by the restriction itself, since an atom
is tested only against the mass of the small cube that contains it.

\subsection{The main theorem}

Fix $\eta_0\in C^\infty([0,\infty))$ such that
$\eta_0=0\text{ on }[0,1/2]$, $\eta_0=1$ on $[1,\infty)$ and $
 0\leq\eta_0\leq1$.
For $0<\varepsilon<R$, define
$
 \chi_{\varepsilon,R}(z)=\eta_0(|z|/\varepsilon)\eta_0(R/|z|) $ for $z\ne0$
  and $
 \chi_{\varepsilon,R}(0)=0,
$
and define the smooth truncations
$$
 T_{\varepsilon,R,\mu}f(x)
 =\int_{\R^n}\chi_{\varepsilon,R}(x-y)\bfK^{\alpha,n}(x,y)f(y)\,d\mu(y)
$$
with integrand zero at $x=y$. Under a uniform norm bound, separability of the two $L^2$ spaces yields
a bounded operator $T$ and strictly increasing integers $j_k$ such that
$$
 T_{2^{-j_k},2^{j_k},\sigma}\to T\quad\text{in the weak operator topology},
 \qquad j_k\to\infty.
$$
Throughout the paper, $T$ denotes an operator obtained in this way.

For a pair $(\mu,\nu)$, let
$$
 \Theta(\mu,\nu)=\sqrt{\mathfrak M_2^\alpha(\mu,\nu)}
 +\mathfrak T(\mu,\nu)+\mathfrak T^{\mathrm{dual}}(\mu,\nu).
$$
Here $\mathfrak M_2^\alpha(\mu,\nu)$ is the sum of the four Muckenhoupt constants, and $\mathfrak T$ and
$\mathfrak T^{\mathrm{dual}}$ are the smooth testing constants, defined in Section~\ref{sec:definitions}.
The one-tailed integrals are taken outside the testing cube. For each measure, we subtract the largest mass at a common atom
in the cube from its total mass on the cube. All six constants use half-open testing cubes $Q$ of arbitrary orientation.
Their edges need not be parallel to the coordinate axes. By \cite[Theorem~1.3]{BoundedDini}, the finiteness of
$\Theta(\sigma,\omega)$ is equivalent to the uniform boundedness of the smooth truncations.

Let $D$ be an axis-parallel half-open cube. We restrict both measures to $D$.
We compute $\Theta(\one_D\sigma,\one_D\omega)$ using all testing cubes $Q$ above. Define
\begin{equation}\label{eq:local-data}
 \mathfrak V(r)=\sup_{\substack{D\text{ axis-parallel}\\0<\ell(D)\leq r}}\Theta( \one_D\sigma , \one_D\omega ),
 \qquad \mathfrak V_0=\lim_{r\to0^+}\mathfrak V(r).
\end{equation}
The limit exists by monotonicity.
Common atoms and punctured masses are computed for the restricted pair.
For each testing cube $Q$, both integrals in the testing conditions are taken over
$D\cap Q$, while the one-tailed integrals are taken over $D\setminus Q$.
The essential norm is $$\|T\|_{\ess}=\inf_{K\text{ compact}}\|T-K\|,$$
where $K$ ranges over all compact operators from $L^2(\sigma)$
to $L^2(\omega;\mathbb C^n)$.

\begin{theorem}\label{thm:main}
Let $0\leq\alpha<n$. Suppose that $\sigma,\omega$ are locally finite positive Borel measures, $\sigma$ is supported
on a compact regular $C^{1,\mathrm{Dini}}$ curve $\Gamma$, and $\Theta(\sigma,\omega)<\infty$.
Common atoms are allowed. As $\varepsilon\to0^+$ and $R\to\infty$ independently, the smooth truncations converge
in the weak operator topology to a unique bounded operator $T$, and
\begin{equation}\label{eq:main-essential}
 \|T\|_{\ess}\approx\mathfrak V_0.
\end{equation}
The comparison constants depend only on $n,\alpha,\eta_0$ and the fixed geometric parameters of $\Gamma$
in the boundedness theorem. They are independent of the measures.

Consequently, $T$ is compact if and only if $\mathfrak V(r)\to0$ as $r\to0^+$. In this case,
$$
 \|T_{\varepsilon,R,\sigma}-T\|\to0
 \quad\text{as  }\varepsilon\to0^+,\ R\to\infty,
$$
with the two parameters varying independently.
The same conclusions hold when $\omega$ is the curve-supported measure.
\end{theorem}

The condition $\mathfrak V(r)\to0$ as $r\to0^+$ uses only restrictions to small cubes, and the comparison
\eqref{eq:main-essential} is two-sided. The lower bound $\mathfrak V_0\lesssim\|T\|_{\ess}$ holds for an
arbitrary compactly supported $\sigma$. The curve hypothesis enters only in the upper bound, through the
boundedness theorem of \cite{BoundedDini}.

\subsection*{Outline of the proof}

The proof has two independent halves: an exact localization of the essential norm that uses nothing about the
curve, and a comparison of the localized norms with the testing constants that uses the boundedness theorem.

For the first half, let $T_D=\one_DT\one_D$ and
$$
 m_T(r)=\sup_{\substack{D\text{ axis-parallel}\\0<\ell(D)\leq r}}\|T_D\|,
 \qquad m_T(0)=\lim_{r\to0^+}m_T(r).
$$
When the source measure has compact support, we prove the exact identity
$$
 \|T\|_{\ess}=m_T(0).
$$
For classical local principles, see Simonenko~\cite[Section~5, Lemma~2.5]{Simonenko}.
When $\sigma\ne0$, boundedness implies $\int_{|x|>A}|x|^{-2d}\,d\omega(x)<\infty$ for sufficiently large $A$,
so the kernel away from the diagonal is Hilbert--Schmidt even when $\omega$ has unbounded support.
Decomposing into neighboring cubes, we prove that commutators with continuous multipliers are compact, and a finite
partition whose squares sum to one on $\supp\sigma$ yields $\|T\|_{\ess}\leq m_T(r)$.
For the reverse inequality, common atoms prevent a direct weak-convergence argument. Since $T$ vanishes on the
atomic diagonal and the mixed terms have fixed finite-rank factors, removing the limiting atom $p$ from cubes
$D_j$ shrinking to $p$ changes $T_{D_j}$ by $o(1)$ in norm, and the projections with $p$ removed do tend strongly
to zero. Comparing with an arbitrary compact operator proves $m_T(0)\leq\|T\|_{\ess}$.

For the second half, weak convergence bounds $\|T_D\|$ by the supremum of the truncation norms, and the Fourier
representation of the cutoff recovers each smooth truncation from $T_D$ with a Fourier $L^1$ bound independent of
both truncation radii:
$$
 \|T_D\|\leq\sup_{0<\varepsilon<R}
       \|T_{\varepsilon,R, \one_D\sigma }\|_{L^2( \one_D\sigma )\to L^2( \one_D\omega ;\mathbb C^n)}
 \leq C_\eta\|T_D\|.
$$
By the companion boundedness theorem, the middle quantity is comparable to $\Theta( \one_D\sigma , \one_D\omega )$
uniformly in $D$. Together with $\|T\|_{\ess}=m_T(0)$, this proves \eqref{eq:main-essential}.

Finally, since the cutoffs vanish on the common-atom diagonal, they cannot converge in norm to a compact operator
with a nonzero atomic diagonal. Removing this diagonal from compact approximants to $T$ proves norm convergence of
the smooth truncations when $T$ is compact.

\subsection{An exact essential norm}

Two-weight theory rarely yields more than two-sided estimates with unspecified constants. In the principal case
$\alpha=n-1$, when both measures are positive continuous multiples of arclength on the same closed curve, we obtain
the essential norm exactly, and the formula identifies it as a purely local, tangential quantity.
Let $ds_\Gamma$ denote arclength. For $\alpha=n-1$, define the Euclidean truncation by
$$
 T^\sharp_\varepsilon f(x)
 =c_{n-1,n}\int_{\substack{y\in\Gamma\\|x-y|>\varepsilon}}
       \frac{y-x}{|y-x|^2}f(y)\,d\sigma(y).
$$

\begin{theorem}\label{thm:tangent-exact}
Let $n\geq2$, $\alpha=n-1$, and let $\Gamma$ be a closed regular $C^{1,\mathrm{Dini}}$ curve.
Suppose $d\sigma=a\,ds_\Gamma$ and $d\omega=b\,ds_\Gamma$, where $a,b$ are positive continuous functions on $\Gamma$.
The operators $T^\sharp_\varepsilon$ converge strongly from $L^2(\sigma)$ to $L^2(\omega;\mathbb C^n)$ as
$\varepsilon\to0^+$. Their limit $T$ is the almost-everywhere Euclidean principal value and is also the strong limit
of $T_{\varepsilon,R,\sigma}$ as $\varepsilon\to0^+$ and $R\to\infty$. Moreover,
\begin{equation}\label{eq:tangent-exact}
 \|T\|_{\ess}=\pi c_{n-1,n}\max_{x\in\Gamma}\sqrt{a(x)b(x)}.
\end{equation}
\end{theorem}

The factor $\pi$ is the norm of the Hilbert transform on $L^2$ with kernel $1/(t-s)$.
For an arclength parametrization $\Phi$, tangent $\tau=\Phi'$, and derivative modulus $\vartheta$, the local estimate
$$
 \left|\frac{\Phi(t)-\Phi(s)}{|\Phi(t)-\Phi(s)|^2}-\frac{\tau(s)}{t-s}\right|
 \leq C\frac{\vartheta(|t-s|)}{|t-s|}
$$
shows that the curve kernel is the Hilbert kernel along the tangent plus a remainder, and the Dini condition
$\int_0\vartheta(t)\,dt/t<\infty$ is exactly the integrability of this remainder near the diagonal. It is therefore
the Dini condition, rather than H\"older continuity of the tangent, that makes the remainder compact and reduces the
essential norm to that of a multiplied Hilbert transform. This reduction is also contained in
\cite[Theorem~2.4]{HMT}, with dimension $1$ and codimension $n-1$.
We prove it directly with the present normalization and compare Euclidean, arclength and smooth
truncations. Modulated functions supported near a maximum of $\sqrt{ab}$ prove the matching lower bound.

By \cref{cor:tangent-local}, the value in \eqref{eq:tangent-exact} also equals $m_T(0)$ and the largest norm among
the tangent-line operators at points of $\Gamma$. For constant densities it is independent of the length and shape
of the curve. Every continuous normal component is compact, whereas the tangential component has
the same essential norm as the vector, see \cref{cor:tangent-projections}.
For the same densities, \cref{cor:tangent-orders} shows that the fractional Riesz vector is compact when
$\alpha>n-1$, while no bounded operator agrees with its off-diagonal kernel when $\alpha<n-1$. Thus $\alpha=n-1$ is
the only order at which the operator is bounded but not compact, and \eqref{eq:tangent-exact} measures exactly how
far it is from compact.

\subsection{Distance powers and field operators}

The abstract theory is tested on the operators that arise when a curve carries a charge or a current.
Let $\Gamma$ now be a nontrivial compact regular simple $C^1$ curve in $\R^n$, $n\geq2$, and let the source measure
be $ds_\Gamma$. For $\beta>-(n-1)$, the target measure $\dist(x,\Gamma)^\beta\,dx$ is locally finite.
We prove the following compactness ranges in \cref{thm:compact-distance-powers,prop:compact-biot-savart}:
\begin{itemize}
\item On a sufficiently small bounded tube, the Newtonian gradient is compact if and only if $\beta>n-3$.
The Dirichlet Green gradient has the same compactness range when the closed tube lies inside a bounded $C^2$ domain,
see \cref{prop:green}.
\item On $\R^n$, the Newtonian gradient is compact if and only if $n-3<\beta<n-2$.
\item On $\R^3$, the Biot--Savart operator on arbitrary vector densities is compact if and only if $0<\beta<1$.
\end{itemize}

These are precisely the boundedness ranges in \cite[Corollaries~8.2--8.4]{BoundedDini}, so for these weights
boundedness and compactness coincide. The compactness is not of the trivial kind. The local Green and Newtonian
gradients are Hilbert--Schmidt exactly when $\beta>n-2$, so on the interval $n-3<\beta\leq n-2$ compactness holds
but is invisible to the Hilbert--Schmidt criterion, and the whole-space operators are never Hilbert--Schmidt. We
use the quantitative near-curve and far-field estimates from \cite{BoundedDini} to construct Hilbert--Schmidt
approximants converging in operator norm, and we show by direct lower bounds that the ranges are sharp. These
distance-power results require only $C^1$ regularity.

The physical content of compactness is the following. For $\beta>n-3$, weak $L^2(ds_\Gamma)$ convergence of
source densities yields strong weighted $L^2$ convergence of the Green gradients on the bounded tube, and
convergence of their weighted energies. For Biot--Savart fields, the corresponding strong convergence holds on the
whole space for $0<\beta<1$, see \cref{cor:weak-source-fields}. In particular, a weakly convergent sequence of
currents along $\Gamma$ produces magnetic fields converging in the weighted energy norm.

Section~\ref{sec:definitions} defines the testing quantities and specifies the boundedness result used here.
Sections~\ref{sec:multipliers}--\ref{sec:criterion} prove \cref{thm:main}.
Section~\ref{sec:examples} discusses testing and examples. Sections~\ref{sec:compact-pde} and \ref{sec:tangent}
prove the field results and \cref{thm:tangent-exact}, respectively.

\section{Testing quantities and the boundedness theorem}\label{sec:definitions}

A compact regular $C^{1,\mathrm{Dini}}$ curve is the image of an arclength parametrization $\Phi:[0,L]\to\R^n$ with
$|\Phi'|=1$, where $\Phi$ is one-to-one except for a possible identification of the two endpoints.  Its derivative
modulus satisfies
$$
 \vartheta_\Phi(t)=\sup_{|u-v|\leq t}|\Phi'(u)-\Phi'(v)|,
 \qquad \int_0^1\vartheta_\Phi(t)\,\frac{dt}{t}<\infty.
$$
In the closed case the endpoint derivatives agree and parameter distance is cyclic.
We assume $\supp\sigma\subset\Gamma$. In particular, $\sigma$ is finite.

Let $\mathcal Q_{\mathrm{rot}} $ consist of the rotated half-open cubes $Q=c+\ell U[-1/2,1/2)^n$, where $c\in\R^n$, $0<\ell<\infty$, and $U$ is
orthogonal, and let $c_Q=c$ and $\ell(Q)=\ell$.
All the constants below use this family.
The cubes $D$ in the suprema are axis-parallel and half-open.
On $L^2(\mu)$, we use the inner product
$\langle f,g\rangle=\int f\,\overline g\,d\mu$.

\subsection{The four Muckenhoupt conditions}
\label{The four Muckenhoupt conditions}
For locally finite measures $\mu,\nu$, define
$$
 \mathcal P^\alpha(Q,\mu)=\int_{\R^n}\frac{\ell(Q)^d}{(\ell(Q)+|y-c_Q|)^{2d}}\,d\mu(y).
$$
The one-tailed constants are
$$
 \mathcal A_2^\alpha(\mu,\nu)
 =\sup_{Q\in\mathcal Q_{\mathrm{rot}} }\frac{\nu(Q)}{\ell(Q)^d}\mathcal P^\alpha(Q,\one_{Q^c}\mu),
 \qquad \mathcal A_2^{\alpha,*}(\mu,\nu)=\mathcal A_2^\alpha(\nu,\mu).
$$
A term with zero inside mass is defined to be zero.

Let $$\mathfrak P(\mu,\nu)=\{p:\mu(\{p\})\nu(\{p\})>0\}$$ be the set of common atoms.
For a bounded Borel set $E$, define
$$
 \mu(E,\mathfrak P)=\mu(E)-\sup_{p\in E\cap\mathfrak P}\mu(\{p\}),
$$
where an empty supremum is zero. Since $\mu(E)<\infty$, a largest mass exists when $E\cap\mathfrak P\ne\varnothing$.
The punctured constants are
$$
 A_2^{\alpha}(\mu,\nu)
 =\sup_{Q\in\mathcal Q_{\mathrm{rot}}}
   \frac{\mu(Q,\mathfrak P)\nu(Q)}{\ell(Q)^{2d}},
 \qquad
 A_2^{\alpha,*}(\mu,\nu)
 =\sup_{Q\in\mathcal Q_{\mathrm{rot}}}
   \frac{\mu(Q)\nu(Q,\mathfrak P)}{\ell(Q)^{2d}}.
$$
Define
$$
 \mathfrak M_2^\alpha(\mu,\nu)=\mathcal A_2^\alpha(\mu,\nu)+\mathcal A_2^{\alpha,*}(\mu,\nu)
 +A_2^{\alpha}(\mu,\nu)+A_2^{\alpha,*}(\mu,\nu).
$$
\subsection{Smooth truncations and testing}\label{subsec:truncations}

We use the radial cutoff $\chi_{\varepsilon,R}$ from the Introduction, as in \cite[Section~3]{BoundedDini}.
It belongs to $C_c^\infty(\R^n)$, equals one for $\varepsilon\leq|z|\leq R$, and vanishes for
$|z|\leq\varepsilon/2$ or $|z|\geq2R$. Define
$$
 \mathfrak N_s(\mu,\nu)=\sup_{0<\varepsilon<R}
   \|T_{\varepsilon,R,\mu}\|_{L^2(\mu)\to L^2(\nu;\mathbb C^n)}.
$$

For a vector kernel we use the componentwise transposed vector $\bfK^*(x,y)=\overline{\bfK(y,x)}$.  Define
$$
 \mathfrak T(\mu,\nu)
 =\sup_{\substack{Q\in\mathcal Q_{\mathrm{rot}} ,\ \mu(Q)>0\\0<\varepsilon<R}}
  \frac{\|\one_QT_{\varepsilon,R,\mu}\one_Q\|_{L^2(\nu)}}{\sqrt{\mu(Q)}},
$$
and
$$
 \mathfrak T^{\mathrm{dual}}(\mu,\nu)
 =\sup_{\substack{Q\in\mathcal Q_{\mathrm{rot}} ,\ \nu(Q)>0\\0<\varepsilon<R}}
  \frac{\|\one_QT_{\varepsilon,R,\nu}^{\vee}\one_Q\|_{L^2(\mu)}}{\sqrt{\nu(Q)}}.
$$
Here $T^\vee$ integrates a complex-valued function against $\bfK^*$. Each component is the adjoint of the
corresponding component of the original operator. Since the Riesz kernel is real and antisymmetric,
$\bfK^*(x,y)=-\bfK^{\alpha,n}(x,y)$, so $T_{\varepsilon,R,\nu}^{\vee}=-T_{\varepsilon,R,\nu}$.
Empty suprema are zero. Let
$$
 \Theta(\mu,\nu)=\sqrt{\mathfrak M_2^\alpha(\mu,\nu)}
 +\mathfrak T(\mu,\nu)+\mathfrak T^{\mathrm{dual}}(\mu,\nu).
$$
Interchanging the measures exchanges the testing constants and the Muckenhoupt constants in pairs, so
$\Theta(\mu,\nu)=\Theta(\nu,\mu)$.

Choose a finite graph cover of $\Gamma$ as in \cite[Lemma~6.2 and Section~7, Step~1]{BoundedDini}. After a rigid motion,
each arc has the form $\{(t,\psi_j(t)):t\in I_j\}$, where $\psi_j$ is $C^{1,\mathrm{Dini}}$. Choose the charts small
enough that their global extensions satisfy $\|D\psi_j\|_\infty<(n-\alpha)/(8n^2)$. Fix the charts, their enlargements and
overlaps, including all slope bounds and derivative Dini moduli, before restricting either measure.

We use the boundedness theorem of \cite[Theorem~1.3 and Lemma~3.3]{BoundedDini} in the following form for smooth truncations.

\begin{proposition}\label{prop:bounded-smooth}
Fix $0\leq\alpha<n$, the curve $\Gamma$, its finite graph cover, and the cutoff $\eta_0$ above.  Let $\mu,\nu$ be
arbitrary locally finite positive Borel measures on $\R^n$, possibly with common atoms.  If $\supp\mu\subset\Gamma$ or
$\supp\nu\subset\Gamma$, then
\begin{equation}\label{eq:bounded-smooth}
 C_\Gamma^{-1}\Theta(\mu,\nu)\leq\mathfrak N_s(\mu,\nu)\leq C_\Gamma\Theta(\mu,\nu).
\end{equation}
The constant $C_\Gamma$ depends only on $n,\alpha,\eta_0$, the graph cover and the geometric parameters in the
boundedness theorem. It is independent of the measures and applies to
$(\one_D\mu,\one_D\nu)$ for every Borel set $D$, with their common atoms recomputed after restriction.
\end{proposition}

\begin{proof}
The four Muckenhoupt constants required in \cite[Theorem~1.3]{BoundedDini} are bounded by
$\mathfrak M_2^\alpha$. That theorem allows either measure to be supported on $\Gamma$.
Denote its tangent-line constants by $\mathfrak N_{\mathrm{tl}}$,
$\mathfrak T_{\mathrm{tl}}$ and $\mathfrak T_{\mathrm{tl}}^{\mathrm{dual}}$,
with the norm and testing constants defined as above using the tangent-line truncations from \cite[Section~3]{BoundedDini}
and suprema over $0<\varepsilon<R$.
Apply \cite[Lemma~3.3]{BoundedDini} to the original Riesz kernel and the family $\mathcal Q_{\mathrm{rot}} $.
The change of variables is the identity, so the quasicubes in that lemma are ordinary cubes. Then
$$
 \mathfrak T_{\mathrm{tl}}+\mathfrak T_{\mathrm{tl}}^{\mathrm{dual}}
 \leq C\bigl(\mathfrak T+\mathfrak T^{\mathrm{dual}}+\sqrt{\mathfrak M_2^\alpha}\bigr)
 =C\Theta(\mu,\nu).
$$
The same lemma and \cite[Theorem~1.3]{BoundedDini} imply
\begin{align*}
 \mathfrak N_s
 &\leq C\bigl(\mathfrak N_{\mathrm{tl}}+\sqrt{\mathfrak M_2^\alpha}\bigr)
 \leq C_\Gamma\bigl(\sqrt{\mathfrak M_2^\alpha}
       +\mathfrak T_{\mathrm{tl}}+\mathfrak T_{\mathrm{tl}}^{\mathrm{dual}}\bigr)
 \leq C_\Gamma\Theta(\mu,\nu),
\end{align*}
with constants independent of both truncation parameters.

For necessity, \cite[Theorem~5(2)]{CurvePaper} applies to the fractional Riesz vector with common atoms.
Our radial cutoff is compactly supported away from zero
and equals one on $[\varepsilon,R]$. The cubes are half-open, the one-tailed Poisson integrals
use $Q^c$ as in \cite[Definition~21]{CurvePaper}, and the punctured mass subtracts the largest common-atom mass
as in \cite[Section~2.2.1]{CurvePaper}.
For $U\in O(n)$, apply it to the pushforward pair $(U_\#\mu,U_\#\nu)$ and axis-parallel cubes.
The constants are independent of $U$, and rotational covariance yields
$$
 \sqrt{\mathfrak M_2^\alpha(\mu,\nu)}
 \leq C_{n,\alpha}\sup_{U\in O(n)}\mathfrak N_s(U_\#\mu,U_\#\nu)
 =C_{n,\alpha}\mathfrak N_s(\mu,\nu).
$$
For the testing constants, restriction is contractive and the $n$ adjoint components each have norm at most
$\mathfrak N_s(\mu,\nu)$. Thus
$$
 \mathfrak T(\mu,\nu)\leq\mathfrak N_s(\mu,\nu),\qquad
 \mathfrak T^{\mathrm{dual}}(\mu,\nu)\leq\sqrt n\,\mathfrak N_s(\mu,\nu),
$$
and consequently
$$
 \Theta(\mu,\nu)\leq\bigl(C_{n,\alpha}+1+\sqrt n\bigr)\mathfrak N_s(\mu,\nu).
$$

Restriction to $D$ changes neither the charts nor the cutoff. The same proof therefore applies to
$(\one_D\mu,\one_D\nu)$ with $C_\Gamma$ independent of $D$ and punctured masses computed for the restricted pair.
\end{proof}

\subsection{Restriction of both measures}

If $\Theta(\sigma,\omega)<\infty$, \cref{prop:bounded-smooth} yields uniformly bounded smooth truncations.
As in the Introduction, fix $T$ and a strictly increasing subsequence $j_k$.

For a Borel set $D$ and a testing cube $Q$, let
$E_Q=D\cap Q$. Then
\begin{align*}
 \mathcal A_2^\alpha( \one_D\sigma , \one_D\omega )
 &=\sup_{Q\in\mathcal Q_{\mathrm{rot}} }
  \omega(E_Q)\int_{D\setminus Q}
     \frac{d\sigma(y)}{(\ell(Q)+|y-c_Q|)^{2d}},\\
 \mathcal A_2^{\alpha,*}( \one_D\sigma , \one_D\omega )
 &=\sup_{Q\in\mathcal Q_{\mathrm{rot}} }
  \sigma(E_Q)\int_{D\setminus Q}
     \frac{d\omega(y)}{(\ell(Q)+|y-c_Q|)^{2d}}.
\end{align*}
The punctured conditions become
\begin{align*}
 A_2^{\alpha}( \one_D\sigma , \one_D\omega )
 &=\sup_{Q\in\mathcal Q_{\mathrm{rot}} }
   \frac{\sigma(E_Q,\mathfrak P(\sigma,\omega)\cap D)\omega(E_Q)}{\ell(Q)^{2d}},\\
 A_2^{\alpha,*}( \one_D\sigma , \one_D\omega )
 &=\sup_{Q\in\mathcal Q_{\mathrm{rot}} }
   \frac{\sigma(E_Q)\omega(E_Q,\mathfrak P(\sigma,\omega)\cap D)}{\ell(Q)^{2d}}.
\end{align*}
We note that 
$$
 \mathfrak T( \one_D\sigma , \one_D\omega )^2
 =\sup_{\substack{Q\in\mathcal Q_{\mathrm{rot}} ,\ \sigma(E_Q)>0\\
                  0<\varepsilon<R}}
 \frac{1}{\sigma(E_Q)}
 \int_{E_Q}
 \left|\int_{E_Q}
  \chi_{\varepsilon,R}(x-y)\bfK^{\alpha,n}(x,y)\,d\sigma(y)
 \right|^2d\omega(x).
$$
The dual expression interchanges the measures and transposes the components.

\section{Weak limits and smooth kernel multipliers}\label{sec:multipliers}

Let $\sigma,\omega$ be positive locally finite Borel measures. Assume that $\Sigma=\supp\sigma$ is compact and
$\mathfrak N_s(\sigma,\omega)<\infty$.

For a Borel set $E$, let $\one_E^\mu f=\one_Ef$ on $L^2(\mu)$, acting componentwise on vector-valued functions.
Fix $T$ and the subsequence $j_k$ chosen above. For every Borel set $D$, define
$$
 T_D=\one_D^\omega T \one_D^\sigma.
$$
Its norm is unchanged on the restricted $L^2$ spaces, and weak convergence implies
$\|T_D\|\leq\mathfrak N_s( \one_D\sigma , \one_D\omega )$.

For bounded Borel sets $A,B$ with $\dist(A,B)>0$, the cutoff equals one on $A\times B$ whenever
$\varepsilon<\dist(A,B)$ and $R>\diam(A\cup B)$. Passing to any weak limit as $\varepsilon\to0^+$ and $R\to\infty$ yields
\begin{equation}\label{eq:off-diagonal}
\one_A^\omega T \one_B^\sigma f(x)=\one_A(x)\int_B\bfK^{\alpha,n}(x,y)f(y)\,d\sigma(y).
\end{equation}
The kernel is square integrable on $A\times B$. The identity holds first for bounded functions and then on $L^2$ by density.
Since every truncation vanishes on the diagonal, every common atom $p$ satisfies
\begin{equation}\label{eq:zero-atomic-diagonal}
 \one_{\{p\}}^\omega T \one_{\{p\}}^\sigma=0.
\end{equation}
The restricted truncations converge weakly to $T_D$ along the same subsequence.

\begin{proposition}\label{prop:unique}
If $\mathfrak N_s(\sigma,\omega)<\infty$, $T$ is unique, and
$$
 T_{\varepsilon,R,\sigma}\to T\quad\text{in the weak operator topology as }\varepsilon\to0^+,\ R\to\infty.
$$
The two parameters vary independently.
\end{proposition}

\begin{proof}
For two weak limits $T,T'$ along sequences with $\varepsilon\to0^+$ and $R\to\infty$, let $X=T-T'$.
By \eqref{eq:off-diagonal}, $\one_A^\omega X\one_B^\sigma=0$ for bounded separated sets.
Given disjoint Borel sets $A,B$ and $f\in L^2(\sigma)$, $g\in L^2(\omega;\mathbb C^n)$,
choose compact sets $A_m\subset A$, $B_m\subset B$ by inner regularity so that
$$
 \|\one_{A\setminus A_m}g\|_{L^2(\omega)}+\|\one_{B\setminus B_m}f\|_{L^2(\sigma)}\to0.
$$
The disjoint compact sets are separated, so
$
 \langle X\one_B^\sigma f,\one_A^\omega g\rangle
 =\lim_m\langle X\one_{B_m}^\sigma f,\one_{A_m}^\omega g\rangle=0.
$
Taking $(A,B)=(E,E^c)$ and $(E^c,E)$ yields
$
 \one_E^\omega X=X\one_E^\sigma$ for every Borel set $E$.

Let $\lambda=\sigma+\omega$, $a=d\sigma/d\lambda$ and $b=d\omega/d\lambda$.
The maps
$$
 U_\sigma f=\sqrt a\,f,\qquad U_\omega g=\sqrt b\,g
$$
are isometries into $L^2(\lambda)$ and $L^2(\lambda;\mathbb C^n)$, respectively.
In particular, $U_\sigma^*h=\one_{\{a>0\}}h/\sqrt a$, with the quotient defined as zero on $\{a=0\}$.
Let $Z=U_\omega XU_\sigma^*$. The projection identity becomes
$$
 \one_E^\lambda Z=Z\one_E^\lambda,\qquad Z=\one_\Sigma^\lambda Z\one_\Sigma^\lambda.
$$
Since $\lambda(\Sigma)<\infty$, define $v=Z\one_\Sigma$. For every Borel set $E\subset\Sigma$,
$$
 Z\one_E=\one_Ev,\qquad
 \int_E|v|^2\,d\lambda\leq\|Z\|^2\lambda(E).
$$
Thus $v\in L^\infty(\lambda;\mathbb C^n)$, and simple-function approximation proves $Z=M_v$.
Here $M_v$ denotes multiplication by $v$. Multiplication by a complex-valued function acts componentwise.
The ranges of $U_\sigma$ and $U_\omega$ consist of functions supported on $\{a>0\}$ and $\{b>0\}$, respectively. Hence
$$
 Z=M_{\one_{\{b>0\}}}ZM_{\one_{\{a>0\}}},\qquad v=0\quad\text{on }\{ab=0\}.
$$

To determine $v$, consider the operators
$$
 Q_{\varepsilon,R}
 =M_{\sqrt a}U_\omega T_{\varepsilon,R,\sigma}U_\sigma^*M_{\sqrt b}.
$$
Their kernels are
$$
 q_{\varepsilon,R}(x,y)
 =\sqrt{a(x)b(x)}\,\chi_{\varepsilon,R}(x-y)\bfK^{\alpha,n}(x,y)\sqrt{a(y)b(y)}.
$$
These kernels are real, bounded and supported in $\Sigma\times\Sigma$, hence square integrable with respect to
$\lambda\times\lambda$. Since $q_{\varepsilon,R}(y,x)=-q_{\varepsilon,R}(x,y)$, every component of
$Q_{\varepsilon,R}$ is real and skew-adjoint. Both properties pass to weak operator limits.
The difference of the limits along the two sequences is
$$
 M_{\sqrt a}ZM_{\sqrt b}=M_{\sqrt{ab}\,v}.
$$
Each coefficient $\sqrt{ab}\,v_i$ is real, so skew-adjointness implies
$$
 M_{\sqrt{ab}\,v_i}=-M_{\sqrt{ab}\,v_i}^*=-M_{\sqrt{ab}\,v_i}.
$$
Hence $\sqrt{ab}\,v=0$. Since $v=0$ on $\{ab=0\}$, we conclude that $Z=0$ and $T=T'$.

If $T_{\varepsilon,R,\sigma}$ failed to converge to $T$ in the weak operator topology, some matrix coefficient would
remain a fixed positive distance from its value at $T$ along a sequence with $\varepsilon\to0^+$ and $R\to\infty$.
Uniform boundedness and separability yield a subsequence converging in the weak operator topology,
whose limit must be $T$ by uniqueness, a contradiction.
\end{proof}

To recover each smooth truncation from $T$, we define a kernel multiplier for arbitrary positive locally finite measures.
For $h\in C_c^\infty(\R^n)$, let
$$
 \widehat h(\xi)=\int_{\R^n}e^{-ix\cdot\xi}h(x)\,dx,
 \qquad \Lambda(h)=(2\pi)^{-n}\|\widehat h\|_1.
$$
Multiplication by $e^{ix\cdot\xi}$ defines a unitary $U_\xi^\mu$ on $L^2(\mu)$, also acting componentwise. For
$S:L^2(\sigma)\to L^2(\omega;\mathbb C^n)$ bounded, define the weak operator integral
\begin{equation}\label{eq:Fourier-Schur}
 \mathfrak S_h(S)=(2\pi)^{-n}\int_{\R^n}\widehat h(\xi)U_\xi^\omega S U_{-\xi}^\sigma\,d\xi.
\end{equation}
For $f\in L^2(\sigma)$ and $g\in L^2(\omega;\mathbb C^n)$,
$$
 |\langle\mathfrak S_h(S)f,g\rangle|
 \leq(2\pi)^{-n}\int_{\R^n}|\widehat h(\xi)|\|S\|\|f\|_2\|g\|_2\,d\xi
 =\Lambda(h)\|S\|\|f\|_2\|g\|_2.
$$
Thus $\|\mathfrak S_h(S)\|\leq\Lambda(h)\|S\|$. Fubini and the Fourier product formula imply
\begin{align*}
 \mathfrak S_{h_1}(\mathfrak S_{h_2}(S))
 &=(2\pi)^{-2n}\iint\widehat h_1(\xi)\widehat h_2(\zeta)
       U_{\xi+\zeta}^\omega S U_{-\xi-\zeta}^\sigma\,d\zeta\,d\xi\\
 &=(2\pi)^{-n}\int\widehat{h_1h_2}(\xi)U_\xi^\omega S U_{-\xi}^\sigma\,d\xi\\
 &=\mathfrak S_{h_1h_2}(S).
\end{align*}
Defining $\mathfrak S_1(S)=S$ extends these identities to constants plus $C_c^\infty$ functions. If $S$ has a
square-integrable kernel $k$, Fourier inversion and Fubini yield
$$
 \mathfrak S_h(S)f(x)=\int h(x-y)k(x,y)f(y)\,d\sigma(y).
$$
Since $U_\xi^\mu$ commutes with $\one_D^\mu$, for every $h$ above,
$$
 \one_D^\omega\mathfrak S_h(S)\one_D^\sigma=\mathfrak S_h(\one_D^\omega S \one_D^\sigma).
$$
If $S_j\to S$ in the weak operator topology and $\sup_j\|S_j\|<\infty$, dominated convergence in
\eqref{eq:Fourier-Schur} yields
\begin{align*}
 \lim_j\langle\mathfrak S_h(S_j)f,g\rangle
 &=(2\pi)^{-n}\int\widehat h(\xi)
        \lim_j\langle U_\xi^\omega S_jU_{-\xi}^\sigma f,g\rangle\,d\xi
 =\langle\mathfrak S_h(S)f,g\rangle,
\end{align*}
with integrable majorant $|\widehat h(\xi)|\sup_j\|S_j\|\|f\|_2\|g\|_2$. The constant term passes to the limit directly.

\begin{lemma}\label{lem:Schur}
For arbitrary positive locally finite measures $\sigma,\omega$, there is a constant $C_\eta$ such that
$$
 \|\mathfrak S_{\chi_{\varepsilon,R}}(S)\|\leq C_\eta\|S\|
 \qquad(0<\varepsilon<R)
$$
for every bounded operator $S:L^2(\sigma)\to L^2(\omega;\mathbb C^n)$.
If $\Sigma=\supp\sigma$ is compact, then
\begin{equation}\label{eq:cutoff-of-limit}
 \mathfrak S_{\chi_{\varepsilon,R}}(T)=T_{\varepsilon,R,\sigma}.
\end{equation}
Consequently, for every Borel set $D$,
\begin{equation}\label{eq:compression-truncation}
 \|T_D\|\leq\mathfrak N_s( \one_D\sigma , \one_D\omega )\leq C_\eta\|T_D\|.
\end{equation}
\end{lemma}

\begin{proof}
Let $a(z)=1-\eta_0(|z|)$ and $b(z)=\eta_0(1/|z|)$ for $z\ne0$, with $b(0)=1$. Both functions belong to
$C_c^\infty(\R^n)$, and
$$
 \chi_{\varepsilon,R}(z)=(1-a(z/\varepsilon))b(z/R).
$$
Since $\Lambda(a(\cdot/\varepsilon))=\Lambda(a)$ and $\Lambda(b(\cdot/R))=\Lambda(b)$,
\begin{align*}
 \|\mathfrak S_{\chi_{\varepsilon,R}}(S)\|
 &=\|\mathfrak S_{b(\cdot/R)}(S)-\mathfrak S_{a(\cdot/\varepsilon)}(\mathfrak S_{b(\cdot/R)}(S))\|
 \leq(1+\Lambda(a))\Lambda(b)\|S\|=:C_\eta\|S\|.
\end{align*}

Let $k_j$ be the kernel of $T_{2^{-j},2^j,\sigma}$. Compactness of $\Sigma$ and local finiteness imply
$$
 \|k_j\|_{L^2(\omega\times\sigma)}^2
 \leq C2^{2jd}\sigma(\Sigma)\,\omega(\{x:\dist(x,\Sigma)\leq2^{j+1}\})<\infty.
$$
For fixed $\varepsilon,R$, choose $k$ large enough that $2^{-j_k}<\varepsilon/2$ and $2^{j_k}>2R$. Then
$
 \chi_{\varepsilon,R}\chi_{2^{-j_k},2^{j_k}}=\chi_{\varepsilon,R}.
$

The kernel formula and weak operator continuity along the defining subsequence yield
$$
 T_{\varepsilon,R,\sigma}
 =\mathfrak S_{\chi_{\varepsilon,R}}(T_{2^{-j_k},2^{j_k},\sigma})
 \to\mathfrak S_{\chi_{\varepsilon,R}}(T),
$$
which proves \eqref{eq:cutoff-of-limit}. Commuting the projections through $\mathfrak S_{\chi_{\varepsilon,R}}$, we obtain
\begin{align*}
 \mathfrak N_s( \one_D\sigma , \one_D\omega )
 &=\sup_{0<\varepsilon<R}\|\one_D^\omega T_{\varepsilon,R,\sigma}\one_D^\sigma\|
 =\sup_{0<\varepsilon<R}\|\mathfrak S_{\chi_{\varepsilon,R}}(T_D)\|
 \leq C_\eta\|T_D\|.
\end{align*}
Weak lower semicontinuity proves the other inequality:
$$
 \|T_D\|\leq\liminf_k\|\one_D^\omega T_{2^{-j_k},2^{j_k},\sigma}\one_D^\sigma\|
 \leq\mathfrak N_s( \one_D\sigma , \one_D\omega ).
$$
The proof is complete.
\end{proof}

\section{Localization of the essential norm}\label{sec:localization}

Assume that $\Sigma:=\supp\sigma$ is compact. Define
$$
 m_T(r)=\sup_{\substack{D\text{ axis-parallel}\\0<\ell(D)\leq r}}\|\one_D^\omega T\one_D^\sigma\|,
 \qquad m_T(0)=\lim_{r\to0^+}m_T(r).
$$
We prove $\|T\|_{\ess}=m_T(0)$ using continuous multipliers for the upper bound and the vanishing of the common-atom diagonal
for the lower bound.

\begin{lemma}
For every $r>0$, the kernel away from the diagonal is square integrable:
\begin{equation}\label{eq:HS-away}
 \iint_{|x-y|\geq r}|\bfK^{\alpha,n}(x,y)|^2\,d\sigma(y)\,d\omega(x)<\infty.
\end{equation}
\end{lemma}

\begin{proof}
The integral is zero if $\sigma=0$. Otherwise choose $R_0\geq1$ with $\Sigma\subset B(0,R_0)$ and let
$M=\sigma(\Sigma)>0$. Apply \eqref{eq:off-diagonal} on $\{2R_0<|x|<N\}\times\Sigma$ and let $N\to\infty$.
For $\omega$-almost every $x$ with $|x|>2R_0$,
\begin{align*}
 |T\one_\Sigma(x)|
 &\geq-\frac{x}{|x|}\cdot\int_\Sigma\bfK^{\alpha,n}(x,y)\,d\sigma(y)
 =c_{\alpha,n}\int_\Sigma\frac{|x|-x\cdot y/|x|}{|x-y|^{d+1}}\,d\sigma(y)\\
 &\geq c_{\alpha,n}\frac{|x|-R_0}{(|x|+R_0)^{d+1}}M
 \geq {c_{n,\alpha}M\over |x|^{d}}.
\end{align*}
Consequently,
$$
 \int_{|x|>2R_0}|x|^{-2d}\,d\omega(x)
 \leq\frac{C_{n,\alpha}}{M^2}\|T\one_\Sigma\|_2^2
 \leq\frac{C_{n,\alpha}}M\|T\|^2<\infty.
$$
Split the integral in \eqref{eq:HS-away} at $|x|=2R_0$. Since $|x-y|\geq|x|/2$ on the outer part,
\begin{align*}
 \iint_{|x-y|\geq r}|\bfK^{\alpha,n}(x,y)|^2\,d\sigma(y)\,d\omega(x)
 &\leq CM r^{-2d}\omega(\overline{B(0,2R_0)})
       +CM\int_{|x|>2R_0}|x|^{-2d}\,d\omega(x)\\
 &\leq CM r^{-2d}\omega(\overline{B(0,2R_0)})+C\|T\|^2\\
 &<\infty.
\end{align*}
The proof is complete.
\end{proof}

For an operator $A$ between the two $L^2$ spaces, define
$[A,M_\phi]=AM_\phi^\sigma-M_\phi^\omega A$, where $M_\phi^\mu f=\phi f$.

\begin{lemma}\label{lem:neighbor}
For every $\phi\in C_c(\R^n)$, the commutator $[T,M_\phi]$ is compact.
\end{lemma}

\begin{proof}
For $r>0$, partition $\R^n$ into the half-open cubes $Q_a=r(a+[0,1)^n)$, $a\in\mathbb Z^n$, and define
$$
 N_{r,v}=\sum_a\one_{Q_{a+v}}^\omega T\one_{Q_a}^\sigma
 \quad(v\in\{-1,0,1\}^n),\qquad
 N_r=\sum_{v\in\{-1,0,1\}^n}N_{r,v}.
$$
Only finitely many cubes $Q_a$ meet $\Sigma$. For fixed $v$, both families of cubes are disjoint, and hence
$$
 \|N_{r,v}f\|_2^2
 =\sum_a\|\one_{Q_{a+v}}^\omega T\one_{Q_a}^\sigma f\|_2^2
 \leq\|T\|^2\sum_a\|\one_{Q_a}^\sigma f\|_2^2
 =\|T\|^2\|f\|_2^2.
$$
For the remaining cube pairs, define $H_r$ by the kernel
$$
 h_r(x,y)=\sum_{|a-b|_\infty\geq2}
     \one_{Q_b}(x)\bfK^{\alpha,n}(x,y)\one_{Q_a}(y).
$$
The products $Q_b\times Q_a$ in this sum are disjoint and lie in $\{|x-y|\geq r\}$. By \eqref{eq:HS-away},
$$
 \|H_r\|_{\HS}^2
 =\iint|h_r(x,y)|^2\,d\sigma(y)\,d\omega(x)
 \leq\iint_{|x-y|\geq r}|\bfK^{\alpha,n}(x,y)|^2\,d\sigma(y)\,d\omega(x)<\infty.
$$
For bounded, compactly supported $f,g$, the sums below are finite, and \eqref{eq:off-diagonal} implies
\begin{align*}
 \langle(T-N_r)f,g\rangle
 &=\sum_{|a-b|_\infty\geq2}\langle \one_{Q_b}^\omega T\one_{Q_a}^\sigma f,g\rangle\\
 &=\sum_{|a-b|_\infty\geq2}\int_{Q_b}\int_{Q_a}
       \bfK^{\alpha,n}(x,y)f(y)\cdot\overline{g(x)}\,d\sigma(y)\,d\omega(x)\\
 &=\langle H_rf,g\rangle.
\end{align*}
Thus $T=H_r+N_r$ by density.

Let $\varpi_\phi(t)=\sup_{|x-y|\leq t}|\phi(x)-\phi(y)|$, which tends to zero as $t\to0^+$.
Let $B_{a,v}=\one_{Q_{a+v}}^\omega T\one_{Q_a}^\sigma$ and $c_a=\phi(ra)$. Since both cubes lie within distance
$2\sqrt n\,r$ of $ra$,
\begin{align*}
 \|[B_{a,v},M_\phi]\|
 &=\|B_{a,v}(M_\phi^\sigma-c_aI)-(M_\phi^\omega-c_aI)B_{a,v}\|
 \leq2\varpi_\phi(2\sqrt n\,r)\|B_{a,v}\|
 \leq2\varpi_\phi(2\sqrt n\,r)\|T\|.
\end{align*}
Multiplication preserves supports. For fixed $v$, the cubes $Q_a$ are disjoint, as are the cubes $Q_{a+v}$, and hence
\begin{align*}
 \|[T,M_\phi]-[H_r,M_\phi]\|
 &\leq\sum_{v\in\{-1,0,1\}^n}\|[N_{r,v},M_\phi]\|
 \leq2\cdot3^n\varpi_\phi(2\sqrt n\,r)\|T\|\to0.
\end{align*}
Each $[H_r,M_\phi]$ is Hilbert--Schmidt, so $[T,M_\phi]$ is compact.
\end{proof}

\begin{theorem}\label{thm:geometric}
If $\sigma$ has compact support, then
$$
 \|T\|_{\ess}=m_T(0).
$$
In particular, $T$ is compact if and only if $m_T(r)\to0$ as $r\to0^+$.
\end{theorem}

\begin{proof}
The case $\sigma=0$ is trivial. Fix $r>0$. Choose finitely many half-open axis-parallel cubes $D_i$,
with $\ell(D_i)\leq r$, and nonnegative $\psi_i\in C_c^\infty(\operatorname{int}D_i)$ such that
$q:=\sum_i\psi_i^2>0$ on a neighborhood of $\Sigma$. Choose $\chi\in C_c^\infty(\{q>0\})$ with
$0\leq\chi\leq1$ and $\chi=1$ on a neighborhood of $\Sigma$. Define
$$
 \phi_i=\frac{\chi\psi_i}{\sqrt q}\quad\text{on }\{q>0\},
 \qquad \phi_i=0\quad\text{elsewhere}.
$$
Then $\phi_i\in C_c^\infty(\operatorname{int}D_i)$ and
$$
 \sum_i\phi_i^2=\chi^2\leq1\quad\text{on }\R^n,
 \qquad \sum_i\phi_i^2=1\quad\text{on }\Sigma.
$$
Define $A_r=\sum_iM_{\phi_i}^\omega TM_{\phi_i}^\sigma$. By \cref{lem:neighbor},
$$
 T-A_r=\sum_i[T,M_{\phi_i}]M_{\phi_i}^\sigma
$$
is compact. The supports of $\phi_i$ and Cauchy--Schwarz imply
\begin{align*}
 \|A_rf\|_2^2
 &=\int\left|\sum_i\phi_i(x)T_{D_i}(\phi_if)(x)\right|^2\,d\omega(x)
 \leq\int\biggl(\sum_i\phi_i(x)^2\biggr)
       \sum_i|T_{D_i}(\phi_if)(x)|^2\,d\omega(x)\\
 &\leq\sum_i\|T_{D_i}(\phi_if)\|_2^2
 \leq m_T(r)^2\sum_i\|\phi_if\|_2^2
 =m_T(r)^2\|f\|_2^2.
\end{align*}
Consequently $\|T\|_{\ess}\leq\|A_r\|\leq m_T(r)$, and letting $r\to0^+$ proves the upper bound.

For the lower bound, the case $m_T(0)=0$ is trivial. If $m_T(0)>0$, choose cubes $D_j$ with
$$
 \ell(D_j)\leq j^{-1},\qquad
 m_T(j^{-1})-j^{-1}\leq\|T_{D_j}\|\leq m_T(j^{-1}).
$$
Since $\sigma(D_j)>0$ for large $j$ and $\Sigma$ is compact, the centers remain bounded. Pass to a subsequence with
$c_{D_j}\to p$, and let $\delta_j=\one_{D_j}(p)$. Define
$$
 A_j=\one_{D_j\setminus\{p\}}^\omega=\one_{D_j}^\omega-\delta_j\one_{\{p\}}^\omega,
 \qquad
 B_j=\one_{D_j\setminus\{p\}}^\sigma=\one_{D_j}^\sigma-\delta_j\one_{\{p\}}^\sigma.
$$
Since every $x\ne p$ eventually lies outside $D_j$, dominated convergence implies $A_j,B_j\to0$ strongly.
The block $\one_{\{p\}}^\omega T\one_{\{p\}}^\sigma$ is zero by \eqref{eq:zero-atomic-diagonal} at a common atom,
otherwise one projection is zero. Expanding $T_{D_j}$ yields
$$
 T_{D_j}-A_jTB_j
 =\delta_j\one_{\{p\}}^\omega TB_j+\delta_jA_jT\one_{\{p\}}^\sigma.
$$
The operators $T^*\one_{\{p\}}^\omega$ and $T\one_{\{p\}}^\sigma$ have finite rank. Since $B_j=B_j^*$,
$$
 \|T_{D_j}-A_jTB_j\|
 \leq\|B_jT^*\one_{\{p\}}^\omega\|+\|A_jT\one_{\{p\}}^\sigma\|\to0.
$$
For compact $K$ and any finite-rank $F$, strong convergence on the range of $F$ shows that
$$
 \limsup_j\|A_jK\|
 \leq\|K-F\|+\lim_j\|A_jF\|=\|K-F\|.
$$
Approximating $K$ by finite-rank operators yields $\|A_jK\|\to0$, and hence
\begin{align*}
 m_T(0)
 &=\lim_j\|T_{D_j}\|=\lim_j\|A_jTB_j\|\\
 &\leq\limsup_j\bigl(\|A_j(T-K)B_j\|+\|A_jKB_j\|\bigr)\\
 &\leq\|T-K\|+\lim_j\|A_jK\|\\
 &=\|T-K\|.
\end{align*}
Taking the infimum over compact $K$ proves the lower bound.
\end{proof}

\section{Essential norm and convergence of truncations}\label{sec:criterion}

\begin{proof}[Proof of \eqref{eq:main-essential}]
Applying \eqref{eq:bounded-smooth} and \eqref{eq:compression-truncation} to $( \one_D\sigma , \one_D\omega )$, we have
$$
 \Theta( \one_D\sigma , \one_D\omega )\approx\mathfrak N_s( \one_D\sigma , \one_D\omega )\approx\|T_D\|,
$$
uniformly in $D$. In particular, $\mathfrak V(r)\lesssim\|T\|$.
Taking suprema over $\ell(D)\leq r$ and then letting $r\to0^+$, we obtain from \cref{thm:geometric}
\begin{align*}
 \mathfrak V_0
 &=\lim_{r\to0^+}\sup_{\ell(D)\leq r}\Theta( \one_D\sigma , \one_D\omega )
 \approx\lim_{r\to0^+}\sup_{\ell(D)\leq r}\|T_D\|
 =m_T(0)=\|T\|_{\ess}.
\end{align*}
The proof is complete.
\end{proof}

For the individual Muckenhoupt and testing bounds, define
$$
 \mathfrak M_{\mathrm{loc}}(r)=\sup_{\ell(D)\leq r}\mathfrak M_2^\alpha( \one_D\sigma , \one_D\omega ),
 \qquad
 \mathfrak T_{\mathrm{loc}}(r)=\sup_{\ell(D)\leq r}\mathfrak T( \one_D\sigma , \one_D\omega ),
$$
and define $\mathfrak T_{\mathrm{loc}}^{\mathrm{dual}}(r)$ in the same way. Then
$$
 \mathfrak V(r)
 \leq\sqrt{\mathfrak M_{\mathrm{loc}}(r)}
       +\mathfrak T_{\mathrm{loc}}(r)+\mathfrak T_{\mathrm{loc}}^{\mathrm{dual}}(r)
 \leq3\mathfrak V(r).
$$
These quantities are monotone in $r$. Their limits at zero satisfy
$$
 \|T\|_{\ess}\approx\sqrt{\mathfrak M_{\mathrm{loc}}(0)}
       +\mathfrak T_{\mathrm{loc}}(0)+\mathfrak T_{\mathrm{loc}}^{\mathrm{dual}}(0).
$$
Thus compactness is equivalent to vanishing of the four local Muckenhoupt constants and the two local testing quantities.

The necessity argument in \cref{prop:bounded-smooth} applies to arbitrary compactly supported $\sigma$.
Combining that argument with \eqref{eq:compression-truncation} yields
$$
 \Theta( \one_D\sigma , \one_D\omega )\leq C_{n,\alpha}\mathfrak N_s( \one_D\sigma , \one_D\omega )
 \leq C_{n,\alpha}C_\eta\|T_D\|.
$$
Taking suprema over $\ell(D)\leq r$ and letting $r\to0^+$, we conclude from \cref{thm:geometric} that
$\mathfrak V_0\leq C_{n,\alpha}C_\eta\|T\|_{\ess}$.
The curve hypothesis is used only in the converse, through the boundedness theorem.

\subsection{Norm convergence in the compact case}

For a bounded $S:L^2(\sigma)\to L^2(\omega;\mathbb C^n)$ and every common atom $p$, \eqref{eq:Fourier-Schur} implies
$$
 \one_{\{p\}}^\omega\mathfrak S_{\chi_{\varepsilon,R}}(S)\one_{\{p\}}^\sigma
 =\chi_{\varepsilon,R}(0)\one_{\{p\}}^\omega S\one_{\{p\}}^\sigma=0.
$$
Thus a nonzero common-atom diagonal prevents these cutoffs from converging to $S$ in norm, even when $S$ is compact.
The common-atom set $\mathfrak P$ is countable. Let $e_p^\sigma=\one_{\{p\}}/\sqrt{\sigma(\{p\})}$ for
$p\in\mathfrak P$. Define the atomic diagonal by
$$
 \mathcal D(S)f=\sum_{p\in\mathfrak P}
       \ip{f}{e_p^\sigma}\one_{\{p\}}^\omega Se_p^\sigma.
$$
Orthogonality in the source and the disjoint target supports imply
$$
 \|\mathcal D(S)f\|^2
 =\sum_{p\in\mathfrak P}|\ip{f}{e_p^\sigma}|^2\|\one_{\{p\}}^\omega Se_p^\sigma\|^2
 \leq\sup_{p\in\mathfrak P}\|\one_{\{p\}}^\omega Se_p^\sigma\|^2\|f\|^2.
$$
Testing on $e_p^\sigma$ proves
\begin{equation}\label{eq:diagonal-contraction}
 \|\mathcal D(S)\|=\sup_{p\in\mathfrak P}\|\one_{\{p\}}^\omega Se_p^\sigma\|\leq\|S\|.
\end{equation}
Moreover, $\one_{\{p\}}^\omega\mathcal D(S)e_p^\sigma=\one_{\{p\}}^\omega Se_p^\sigma$, so
$\mathcal D^2(S)=\mathcal D(S)$. Empty sums and suprema are zero.

\begin{lemma}\label{lem:compact-diagonal}
For compact $K$, both $\mathcal D(K)$ and $K_0=K-\mathcal D(K)$ are compact, and $\mathcal D(K_0)=0$.
For $T$,
\begin{equation}\label{eq:compact-correction}
 \|T-K_0\|\leq2\|T-K\|.
\end{equation}
\end{lemma}

\begin{proof}
Along distinct atoms, $e_p^\sigma\to0$ weakly, so compactness of $K$ implies $\|Ke_p^\sigma\|\to0$.
Given $\varepsilon>0$, choose finite $F\subset\mathfrak P$ such that
$\sup_{p\in\mathfrak P\setminus F}\|Ke_p^\sigma\|\leq\varepsilon$. The finite sum
$\mathcal D_F(K)$ has finite rank, and
$$
 \|\mathcal D(K)-\mathcal D_F(K)\|
 =\sup_{p\in\mathfrak P\setminus F}\|\one_{\{p\}}^\omega Ke_p^\sigma\|
 \leq\sup_{p\in\mathfrak P\setminus F}\|Ke_p^\sigma\|\leq\varepsilon.
$$
Thus $\mathcal D(K)$ and $K_0$ are compact, with
$$
 \mathcal D(K_0)=\mathcal D(K)-\mathcal D^2(K)=0.
$$
Using $\mathcal D(T)=0$ from \eqref{eq:zero-atomic-diagonal} and the contraction \eqref{eq:diagonal-contraction},
we estimate
$$
 \|T-K_0\|
 =\|(T-K)-\mathcal D(T-K)\|
 \leq\|T-K\|+\|\mathcal D(T-K)\|\leq2\|T-K\|.
$$
The proof is complete.
\end{proof}

\begin{lemma}\label{lem:cutoff-compact}
For every compact $K:L^2(\sigma)\to L^2(\omega;\mathbb C^n)$,
$$
 \mathfrak S_{\chi_{\varepsilon,R}}(K)\to K-\mathcal D(K)
 \quad\text{in operator norm}
$$
as $\varepsilon\to0^+$ and $R\to\infty$.
\end{lemma}

\begin{proof}
A finite-rank operator $F$ has a square-integrable kernel
$$
 k_F(x,y)=\sum_{j=1}^N u_j(x)\overline{v_j(y)},
 \qquad u_j\in L^2(\omega;\mathbb C^n),\quad v_j\in L^2(\sigma).
$$
For nonnegative Borel $h$, Tonelli's theorem gives
$$
 \iint_{x=y}h(x,y)\,d\sigma(y)\,d\omega(x)
 =\int h(x,x)\sigma(\{x\})\,d\omega(x)
 =\sum_{p\in\mathfrak P}h(p,p)\sigma(\{p\})\omega(\{p\}).
$$
The diagonal kernel therefore represents $\mathcal D(F)$. Since $\chi_{\varepsilon,R}=1$ for
$\varepsilon\leq|x-y|\leq R$ and vanishes at $x=y$,
\begin{align*}
 &\|\mathfrak S_{\chi_{\varepsilon,R}}(F)-F+\mathcal D(F)\|_{\HS}^2\\
 &=\iint|\chi_{\varepsilon,R}(x-y)-\one_{\{x\ne y\}}|^2|k_F(x,y)|^2\,d\sigma(y)\,d\omega(x)\displaybreak[1]\\
 &\leq\iint_{0<|x-y|<\varepsilon}|k_F(x,y)|^2\,d\sigma(y)\,d\omega(x)
 +\iint_{|x-y|>R}|k_F(x,y)|^2\,d\sigma(y)\,d\omega(x)\\
 &\to0.
\end{align*}
The two integrals tend to zero independently by dominated convergence.

Given compact $K$ and $\delta>0$, choose finite-rank $F$ with $\|K-F\|<\delta$. By \cref{lem:Schur} and
\eqref{eq:diagonal-contraction},
\begin{align*}
 \|\mathfrak S_{\chi_{\varepsilon,R}}(K)-K+\mathcal D(K)\|
 &\leq C_\eta\|K-F\|+\|(K-F)-\mathcal D(K-F)\|
       +\|\mathfrak S_{\chi_{\varepsilon,R}}(F)-F+\mathcal D(F)\|\\
 &\leq(C_\eta+2)\delta+\|\mathfrak S_{\chi_{\varepsilon,R}}(F)-F+\mathcal D(F)\|.
\end{align*}
Let $\varepsilon\to0^+$ and $R\to\infty$ independently, then let $\delta\to0^+$.
\end{proof}

\begin{proposition}\label{prop:truncation-essential}
Each $T_{\varepsilon,R,\sigma}$ is Hilbert--Schmidt. Moreover,
$$
 \|T\|_{\ess}
 \leq\liminf_{j\to\infty}\|T-T_{2^{-j},2^j,\sigma}\|
 \leq\limsup_{j\to\infty}\|T-T_{2^{-j},2^j,\sigma}\|
 \leq2(1+C_\eta)\|T\|_{\ess}.
$$
If $T$ is compact, then $T_{\varepsilon,R,\sigma}\to T$ in operator norm as $\varepsilon\to0^+$ and
$R\to\infty$, with the two parameters varying independently.
\end{proposition}

\begin{proof}
The cutoff restricts $|x-y|$ to $[\varepsilon/2,2R]$. Since $\Sigma=\supp\sigma$ is compact,
$$
 \|T_{\varepsilon,R,\sigma}\|_{\HS}^2
 =\iint|\chi_{\varepsilon,R}(x-y)\bfK^{\alpha,n}(x,y)|^2\,d\sigma(y)\,d\omega(x)
 \leq C\varepsilon^{-2d}\sigma(\Sigma)\omega(\{x:\dist(x,\Sigma)\leq2R\})<\infty.
$$
Hence $\|T\|_{\ess}\leq\|T-T_{\varepsilon,R,\sigma}\|$.

For compact $K$, let $K_0=K-\mathcal D(K)$ as in \cref{lem:compact-diagonal}. By \eqref{eq:cutoff-of-limit} and
\eqref{eq:compact-correction},
\begin{align*}
 \|T-T_{\varepsilon,R,\sigma}\|
 &=\|(T-K_0)-\mathfrak S_{\chi_{\varepsilon,R}}(T-K_0)+K_0-\mathfrak S_{\chi_{\varepsilon,R}}(K_0)\|\\
 &\leq(1+C_\eta)\|T-K_0\|+\|K_0-\mathfrak S_{\chi_{\varepsilon,R}}(K_0)\|\\
 &\leq2(1+C_\eta)\|T-K\|+\|K_0-\mathfrak S_{\chi_{\varepsilon,R}}(K_0)\|.
\end{align*}
Taking the upper limit, applying \cref{lem:cutoff-compact} to $K_0$, and then minimizing over compact $K$, we obtain
$$
 \limsup_{\substack{\varepsilon\to0^+\\R\to\infty}}\|T-T_{\varepsilon,R,\sigma}\|
 \leq2(1+C_\eta)\inf_{K\text{ compact}}\|T-K\|=2(1+C_\eta)\|T\|_{\ess}.
$$
The proof is complete.
\end{proof}

\begin{proof}[Proof of \cref{thm:main}]
When $\sigma$ is curve-supported, \cref{prop:unique} proves uniqueness and weak convergence as
$\varepsilon\to0^+$ and $R\to\infty$. By \cref{prop:truncation-essential}, the convergence is in operator norm
when $T$ is compact.

If $\omega$ is curve-supported, let $T=(T_1,\ldots,T_n)$ and $T^\vee=(T_1^*,\ldots,T_n^*)$. For each component,
antisymmetry and the even cutoff imply $(T_{\varepsilon,R,\sigma,j})^*=-T_{\varepsilon,R,\omega,j}$.
Thus, along the same subsequence, the truncations for $(\omega,\sigma)$ converge weakly to $-T^\vee$. The source measure of this pair is supported on the curve.
The componentwise adjoint bounds imply
$\mathfrak N_s(\omega,\sigma)\leq\sqrt n\,\mathfrak N_s(\sigma,\omega)<\infty$.
By \cref{prop:unique}, the truncations for $(\omega,\sigma)$ converge weakly to a unique limit.
Taking componentwise adjoints proves the same conclusion for the original pair. To compare essential norms,
approximate each $T_j$ by a compact operator and use invariance under adjoints:
$$
 \|T^\vee\|_{\ess}
 \leq\bigg(\sum_{j=1}^n\|T_j^*\|_{\ess}^2\bigg)^{1/2}
 =\bigg(\sum_{j=1}^n\|T_j\|_{\ess}^2\bigg)^{1/2}
 \leq\sqrt n\,\|T\|_{\ess}.
$$
Applying the same inequality to $T^\vee$ proves the reverse comparison. Interchanging the measures also exchanges
the four Muckenhoupt quantities in pairs and the two testing constants, so
$$
 \|T\|_{\ess}\approx\|T^\vee\|_{\ess}
 \approx\lim_{r\to0^+}\sup_{\ell(D)\leq r}\Theta( \one_D\omega , \one_D\sigma )
 =\lim_{r\to0^+}\sup_{\ell(D)\leq r}\Theta( \one_D\sigma , \one_D\omega )=\mathfrak V_0.
$$
If $T$ is compact, so is $T^\vee$. Applying \cref{prop:truncation-essential} to $-T^\vee$ for the pair
$(\omega,\sigma)$, we obtain
\begin{align*}
 \|T_{\varepsilon,R,\sigma}-T\|^2
 &\leq\sum_{j=1}^n\|T_{\varepsilon,R,\sigma,j}-T_j\|^2
 =\sum_{j=1}^n\|T_{\varepsilon,R,\omega,j}+T_j^*\|^2
 \leq n\|T_{\varepsilon,R,\omega}+T^\vee\|^2\to0.
\end{align*}
The proof is complete.
\end{proof}

\section{Testing and examples}\label{sec:examples}

\subsection{Large-cube testing}

For distinct points $a,b$ on a line segment, let $\sigma=\delta_a$ and $\omega=\delta_b$. Then
$$
 Tf(b)=\bfK^{\alpha,n}(b,a)f(a),
$$
so $T$ has rank one. For every cube $Q$ containing both points, the normalized forward and dual testing quantities have
supremum $|\bfK^{\alpha,n}(b,a)|>0$ over truncations, independently of $\ell(Q)$.
If $\ell(D)<|a-b|/\sqrt n$, one restricted measure is zero, so $\Theta( \one_D\sigma , \one_D\omega )=0$ and $\mathfrak V_0=0$.

Large-cube testing quantities also fail to vanish for arclength measures on two short separated arcs. The bounded kernel
defines a Hilbert--Schmidt operator. Choose the arcs so that one kernel component has a fixed nonzero sign.
Then the normalized testing quantity is nonzero on every cube containing both supports.

\subsection{Projected testing conditions}

For $T=(T_1,\ldots,T_n)$, let $T^\vee=(T_1^*,\ldots,T_n^*)$.

\begin{proposition}\label{prop:projected-tests}
Let $E_M^\mu$ be increasing finite-rank orthogonal projections converging strongly to the identity on $L^2(\mu)$,
applied componentwise to vector-valued functions. Then
\begin{align*}
 \limsup_M\sup_{\substack{Q\in\mathcal Q_{\mathrm{rot}} \\\sigma(Q)>0}}
 \frac{\|(I-E_M^\omega)T\one_Q\|_2}{\sqrt{\sigma(Q)}}
 &\leq\|T\|_{\ess},\\
 \limsup_M\sup_{\substack{Q\in\mathcal Q_{\mathrm{rot}} \\\omega(Q)>0}}
 \frac{\|(I-E_M^\sigma)T^\vee\one_Q\|_2}{\sqrt{\omega(Q)}}
 &\leq\sqrt n\,\|T\|_{\ess}.
\end{align*}
In particular, both testing suprema tend to zero when $T$ is compact.
\end{proposition}

\begin{proof}
For compact $K$, the operator $K^\vee=(K_1^*,\ldots,K_n^*)$ is compact, and
\begin{align*}
 \|(I-E_M^\omega)T\|
 &\leq\|T-K\|+\|(I-E_M^\omega)K\|,\\
 \|(I-E_M^\sigma)T^\vee\|
 &\leq\sqrt n\,\|T-K\|+\|(I-E_M^\sigma)K^\vee\|.
\end{align*}
Strong convergence is uniform on compact subsets, so the final terms tend to zero.
The testing suprema are bounded by the left sides. Taking upper limits and then the infimum over compact $K$
proves the estimates.
\end{proof}

For a fixed dyadic grid, let $\langle f\rangle_Q^\mu=\mu(Q)^{-1}\int_Qf\,d\mu$ when $\mu(Q)>0$, and zero otherwise.
Define
$$
 \Delta_Q^\mu f
 =\sum_{Q'\in\operatorname{ch}(Q)}
       \bigl(\langle f\rangle_{Q'}^\mu-\langle f\rangle_Q^\mu\bigr)\one_{Q'}.
$$
Let $H^\mu$ be the orthogonal projection onto the closed span of these difference ranges. In general, $H^\mu$ may differ from $I$.
By their pairwise orthogonality, for increasing finite families $\mathcal F_M$ whose union is the grid,
$$
 S_M^\mu=H^\mu-\sum_{Q\in\mathcal F_M}\Delta_Q^\mu\to0
 \quad\text{strongly}.
$$
If $T$ is compact, orthogonality and \cref{prop:projected-tests} give
$$
 \sup_{\substack{R\text{ dyadic}\\\sigma(R)>0}}
 \frac1{\sigma(R)}
 \sum_{\substack{Q\subset R\\Q\notin\mathcal F_M}}
       \|\Delta_Q^\omega T\one_R\|_2^2
 \leq\|S_M^\omega T\|^2\to0.
$$
The same estimate for $T^\vee$ proves the dual projected condition.

\subsection{Punctured masses and examples}

For any compactly supported $\sigma$, taking $Q=D$ in the punctured
condition and applying the necessary bound from Section~\ref{sec:criterion} yields
$$
 \frac{\sqrt{\sigma(D,\mathfrak P(\sigma,\omega)\cap D)\omega(D)}}{\ell(D)^d}
 \leq\Theta( \one_D\sigma , \one_D\omega )\leq C_{n,\alpha}C_\eta\|T_D\|.
$$
Hence, for every sequence of axis-parallel cubes with $\ell(D_j)\to0$,
$$
 \limsup_j\frac{\sqrt{\sigma(D_j,\mathfrak P(\sigma,\omega)\cap D_j)\omega(D_j)}}{\ell(D_j)^d}
 \leq C_{n,\alpha}C_\eta\|T\|_{\ess}.
$$
The same assertion holds when the largest common-atom mass is subtracted from $\omega(D_j)$. When $d=1$ and both measures dominate positive multiples
of arclength on a common $C^1$ arc, these bounds exclude compactness even in the presence of atoms.
Indeed, removing a point does not change arclength measure, and cubes centered at an interior point have arclength comparable
to their side length.

For $\Gamma=[0,1]e_1$, $\sigma=\omega=ds_\Gamma$ and $\alpha=n-1$, we have $d=1$. The first component is the
finite-interval Hilbert transform, up to normalization, and the other components are zero. Extend functions by zero to
$\R$, and define $Hf(x)=\operatorname{p.v.}\int_\R f(y)/(x-y)\,dy$. Since $\chi_{\varepsilon,R}$ vanishes near zero,
$$
 \chi_{\varepsilon,R}(t)\operatorname{p.v.}\frac1t
 =\frac{\chi_{\varepsilon,R}(t)}t
 \quad\text{in }\mathcal S'(\R).
$$
Fourier inversion, first paired with Schwartz functions, identifies the truncated Hilbert transform with
$\mathfrak S_{\chi_{\varepsilon,R}}(H)$. Applying \cref{lem:Schur} in dimension one and restricting to the segment yields
$$
 \mathfrak N_s(\sigma,\omega)\leq c_{n-1,n}C_\eta\|H\|<\infty.
$$
Hence $\Theta(\sigma,\omega)<\infty$ by \cref{prop:bounded-smooth}. For small axis-parallel cubes $D$ centered at an
interior point of $\Gamma$, the measures have no atoms and
$$
 \Theta( \one_D\sigma , \one_D\omega )^2\geq A_2^{\alpha}( \one_D\sigma , \one_D\omega )
 \geq\frac{\sigma(D)\omega(D)}{\ell(D)^2}=1.
$$
Thus $\mathfrak V_0\geq1$, and \cref{thm:main} implies noncompactness.

For $0<d<1$, let $T$ denote the integral operator with kernel $\bfK^{\alpha,n}$ on the same segment.
The operator $H_r$ with kernel $\one_{\{|x-y|\geq r\}}\bfK^{\alpha,n}(x,y)$ is Hilbert--Schmidt, and the Schur test yields
$$
 \|T-H_r\|\leq2c_{\alpha,n}\int_0^r t^{-d}\,dt
 =\frac{2c_{\alpha,n}}{1-d}r^{1-d}\to0.
$$
Hence $T$ is compact.

For $d=1$, consider the two measures on $\Gamma=[-1,1]e_1$ defined by
$$
 a(t)=\frac{\one_{(-1,0)}(t)}{\sqrt{\log(e/|t|)}},\qquad b(t)=a(-t),\qquad
 d\sigma=a(t)\,dt,\quad d\omega=b(t)\,dt,
$$
with both densities zero at $t=0$. 
Fix an axis-parallel cube $D$.
For every rotated cube $Q$, its intersection with the line has length at most $\sqrt n\,\ell(Q)$, and
$$
 \int_\R\frac{dt}{(\ell(Q)+|te_1-c_Q|)^2}\leq\frac2{\ell(Q)}.
$$
Therefore, 
\[
\mathfrak M_2^\alpha(\one_D\sigma,\one_D\omega) \le C_n\|\one_Da\|_\infty \|\one_Db\|_\infty.
\]
The smooth Hilbert truncations $H_{\varepsilon,R}$ satisfy
$\sup_{\varepsilon<R}\|H_{\varepsilon,R}\|\leq C_\eta\pi$ by the Fourier estimate above.
For $E=D\cap Q\cap\Gamma$, the forward testing integral satisfies
\begin{align*}
 \int_E|c_{n-1,n}H_{\varepsilon,R}(a\one_E)|^2b\,dt
 &\leq c_{n-1,n}^2\|\one_Db\|_\infty\|H_{\varepsilon,R}(a\one_E)\|_2^2\\
 &\leq c_{n-1,n}^2C_\eta^2\pi^2\|\one_Db\|_\infty\int_Ea^2\,dt\\
 &\leq C\|\one_Da\|_\infty\|\one_Db\|_\infty\sigma(E).
\end{align*}
Interchanging $a,b$ gives the corresponding dual estimate, so $\Theta( \one_D\sigma , \one_D\omega )\leq C\|\one_Da\|_\infty^{1/ 2}\|\one_Db\|_\infty^{1/2}$.
The same estimates without restriction show $\Theta(\sigma,\omega)<\infty$.
If $D$ meets both supports and $\ell(D)\leq r$, every point in $D\cap\Gamma$ has $|t|\leq\sqrt n\,r$.
For sufficiently small $r$, taking the supremum over $\ell(D)\leq r$ yields
$$
 \mathfrak V(r)\leq C\bigl(\log(e/r)\bigr)^{-1/2}\to0.
$$
Thus \cref{thm:main} proves compactness.

The squares $[2^{-k},2^{1-k}]^2$, $k\geq2$, have disjoint interiors, and $a(-u)b(v)$ is comparable to $1/k$ on each.
Thus
\begin{align*}
 \int_0^1\int_0^1\frac{a(-u)b(v)}{(u+v)^2}\,du\,dv
 &\geq\sum_{k=2}^\infty
   \int_{2^{-k}}^{2^{1-k}}\int_{2^{-k}}^{2^{1-k}}
   \frac{a(-u)b(v)}{(u+v)^2}\,du\,dv
 \geq C\sum_{k=2}^\infty\frac1k=\infty.
\end{align*}
A Hilbert--Schmidt kernel representing $T$ would agree with the Riesz kernel on every pair of separated compact
subintervals, hence almost everywhere on the product of the two open half-segments.
The divergent integral therefore shows that $T$ is not Hilbert--Schmidt.

\section{Green gradients and Biot--Savart fields}\label{sec:compact-pde}

For the distance power measures in \cite[Corollaries~8.2--8.4]{BoundedDini}, we prove compactness throughout the
boundedness range and determine the Hilbert--Schmidt range. The curve need only be $C^1$. Applying
\cref{thm:main} requires $C^{1,\mathrm{Dini}}$ regularity.

\subsection{Green gradients}

Let $n\geq2$, let $\Omega\subset\R^n$ be a bounded $C^2$ domain, and let $\Gamma\Subset\Omega$ be a compact regular
$C^1$ curve. Let $\sigma$ be a finite positive Borel measure supported on $\Gamma$, and let $\omega$
be a finite positive Borel measure on $\overline\Omega$.

By \cite[Section~8.1]{BoundedDini}, the Dirichlet Green function satisfies
$$
 \nabla_xG_\Omega(x,y)=a_n\bfK^{1,n}(x,y)+B_\Omega(x,y),
 \qquad \sup_{\overline\Omega\times\Gamma}|B_\Omega|<\infty,
$$
where $a_n=(|\mathbb S^{n-1}|c_{1,n})^{-1}$. The bound on $B_\Omega$ depends only on $\Omega$ and
$\dist(\Gamma,\partial\Omega)$, since the boundary data for the harmonic correction are uniformly smooth.
For $\bfK^{1,n}$ and $(\sigma,\omega)$, assume that $\mathfrak N_s(\sigma,\omega)<\infty$.
Let $T$ be the weak operator limit of the smooth truncations. Define
$$
 S_\Omega f(x)=\int_\Gamma B_\Omega(x,y)f(y)\,d\sigma(y),
 \qquad \mathcal G_\Omega=a_nT+S_\Omega.
$$
For $f\in L^2(\sigma)$, consider the potential
$$
 u_f(x)=\int_\Gamma G_\Omega(x,y)f(y)\,d\sigma(y).
$$
Since $\sigma$ is finite, $f\in L^1(\sigma)$.
On compact subsets of $\Omega\setminus\Gamma$,
$G_\Omega$ and $\nabla_xG_\Omega$ are bounded uniformly
for $y\in\Gamma$, so differentiation under the integral
sign is valid. The off-diagonal identity
\eqref{eq:off-diagonal} and the decomposition of
$\nabla_xG_\Omega$ therefore give
\begin{align*}
 \nabla u_f(x)
 &=\int_\Gamma \nabla_xG_\Omega(x,y)f(y)\,d\sigma(y)
 =a_nTf(x)+S_\Omega f(x)=\mathcal G_\Omega f(x)
\end{align*}
for $\omega$-almost every $x\in\overline\Omega\setminus\Gamma$.
The gradient on $\partial\Omega$ is defined by continuous
extension, since $\Gamma\Subset\Omega$.
The equality is understood $\omega$-almost everywhere
because $Tf$ is defined as an element of
$L^2(\omega;\mathbb C^n)$.
When $\omega(\Gamma)>0$, we use $\mathcal G_\Omega$ as the
operator realization of the Green gradient on $\Gamma$.
The identity with the classical gradient above is asserted only
on $\overline\Omega\setminus\Gamma$.
This potential satisfies $-\Delta u_f=f\sigma$ in distributions
and vanishes on the boundary.

\begin{proposition}\label{prop:green}
Under these assumptions,
$$
 \|\mathcal G_\Omega\|_{\mathrm{ess}}
 =|a_n|\|T\|_{\mathrm{ess}}.
$$
Thus $\mathcal G_\Omega$ is compact if and only if $T$ is compact. For $C^{1,\mathrm{Dini}}$ curves,
\cref{thm:main} characterizes compactness of the Green gradient.
\end{proposition}

\begin{proof}
The correction is Hilbert--Schmidt, since
$$
 \|S_\Omega\|_{\HS}^2
 =\int_{\overline\Omega}\int_\Gamma|B_\Omega(x,y)|^2\,d\sigma(y)\,d\omega(x)
 \leq\|B_\Omega\|_\infty^2\sigma(\Gamma)\omega(\overline\Omega)<\infty.
$$
Consequently $\|a_nT+S_\Omega\|_{\ess}=|a_n|\|T\|_{\ess}$.
\end{proof}

\subsection{Distance powers near a curve}

Let $\Gamma$ be a compact regular simple $C^1$ curve, closed or with two endpoints, and let $\sigma=ds_\Gamma$ be
arclength. Define
$$
 d_\Gamma(x)=\dist(x,\Gamma),\qquad U_{\rho_0}=\{x:d_\Gamma(x)<\rho_0\},\qquad d\omega_\beta(x)=d_\Gamma(x)^\beta\,dx.
$$
Assume $\beta>-(n-1)$, so $\omega_\beta$ is locally finite. Choose $\rho_0>0$ small. For Green gradients, require
$\overline U_{\rho_0}\subset\Omega$, where $\Omega$ is a bounded $C^2$ domain.

For $f\in L^2(ds_\Gamma)$ and $x\notin\Gamma$, define the ordinary integral
$$
 Tf(x)=\int_\Gamma\bfK^{1,n}(x,y)f(y)\,ds_\Gamma(y),\qquad
 |Tf(x)|\leq c_{1,n}d_\Gamma(x)^{1-n}\sigma(\Gamma)^{1/2}\|f\|_2.
$$
The Newtonian gradient is $a_nT$. Since $\omega_\beta(\Gamma)=0$, values on $\Gamma$ do not affect the target spaces.
Let $\nu$ be a locally finite positive measure with $\nu(\Gamma)=0$ and $\mathfrak N_s(ds_\Gamma,\nu)<\infty$.
Then this ordinary integral agrees with the weak limit of the smooth truncations in $L^2(\nu)$. Indeed,
\eqref{eq:off-diagonal} gives the equality on each $A_m=\{x:|x|\leq m,\ d_\Gamma(x)\geq1/m\}$, and
$\bigcup_{m=1}^\infty A_m=\R^n\setminus\Gamma$.

After decreasing $\rho_0$, the $C^1$ graph cover in \cite[Lemma~6.2]{BoundedDini} consists of finitely many smaller
target neighbourhoods with enlarged source arcs. The coordinates $x=F(s,z)=(s,\psi(s)+z)$ in a target neighbourhood and
$y=(t,\psi(t))$, $t\in J$, on its enlarged source arc satisfy
$$
 |x-y|\approx\bigl(|s-t|^2+|z|^2\bigr)^{1/2},\qquad
 d_\Gamma(x)\approx\bigl(\dist(s,J)^2+|z|^2\bigr)^{1/2}.
$$
Here $\det DF=1$, the arclength density is bounded above and below, and the remaining source is separated from the target neighbourhood.
At an endpoint, extend the graph by its tangent line and the source density by zero. Fix these charts before
choosing the truncation radii. The constants below may depend on the curve and this finite cover.

\begin{theorem}\label{thm:compact-distance-powers}
Let $n\geq2$, let $\Gamma$ be a nontrivial compact regular simple $C^1$ curve, and assume $\beta>-(n-1)$.
The ordinary integral operator
$$
 T:
 L^2(ds_\Gamma)\to L^2(U_{\rho_0},\omega_\beta;\mathbb C^n)
$$
is compact if and only if $\beta>n-3$. If $\Omega$ is bounded and $C^2$ with $\overline U_{\rho_0}\subset\Omega$,
the Green gradient is also compact if and only if $\beta>n-3$. On the whole space,
$$
 T:L^2(ds_\Gamma)\to L^2(\omega_\beta;\mathbb C^n)
 \quad\hbox{is compact if and only if}\quad
 n-3<\beta<n-2.
$$
There exists $\varepsilon_0>0$, depending on the fixed graph cover, such that for $\beta>n-3$ and
$0<\varepsilon<\varepsilon_0$,
$$
 \|\one_{\{d_\Gamma<\varepsilon\}}T\|^2\leq
 \begin{cases}
  C_\beta\varepsilon^{\beta+3-n},&n\geq3,\\
  C_\beta\varepsilon^{\beta+1}
       \bigl(1+\log^2(e/\varepsilon)\bigr),&n=2.
 \end{cases}
$$
In the whole-space range, for sufficiently large $R$,
$$
 \|\one_{\{|x|>R\}}T\|^2
 \leq C_\beta R^{\beta+2-n}.
$$
\end{theorem}

\begin{proof}
For $\beta>n-3$, we estimate the positive operator
$$
 Pf(x)=\int_\Gamma|x-y|^{1-n}f(y)\,ds_\Gamma(y),
$$
since $|Tf|\leq c_{1,n}P|f|$. On the interior part of a chart, let $I\subset J$ be the target interval and
$\Gamma_J=\{(t,\psi(t)):t\in J\}$. Extend $h(t)=\one_J(t)|f(t,\psi(t))|\sqrt{1+|\psi'(t)|^2}$
by zero and, for $\rho=|z|>0$, define
$$
 k_\rho(t)=\one_{[-C,C]}(t)(t^2+\rho^2)^{(1-n)/2}.
$$
The distance comparison and Young's inequality yield
$$
 \|P(\one_{\Gamma_J}|f|)(F(\cdot,z))\|_{L^2(I)}
 \leq C\|k_\rho*h\|_{L^2(\R)}
 \leq C\|k_\rho\|_1\|h\|_2\leq C\|k_\rho\|_1\|f\|_2.
$$

In an endpoint chart with $J=[0,L_J]$, let $\delta=(s^2+|z|^2)^{1/2}$ for $s<0$.
Cauchy--Schwarz and the graph distance comparison imply
$$
 |P(\one_{\Gamma_J}|f|)(F(s,z))|^2
 \leq C\|f\|_2^2\int_0^\infty(t^2+\delta^2)^{1-n}\,dt
 =C_n\delta^{3-2n}\|f\|_2^2.
$$
Since $d_\Gamma(F(s,z))\approx\delta$, radial integration in $(s,z)$ contributes the factor
$\delta^\beta\delta^{3-2n}\delta^{n-1}=\delta^{\beta+2-n}$. The separated source has bounded kernel, and
$\omega_\beta(\{d_\Gamma<\varepsilon\})\leq C_\beta\varepsilon^{\beta+n-1}$ in these coordinates.
Combining the interior, endpoint and separated-source estimates over the finite cover, we obtain
\begin{align*}
 \|\one_{\{d_\Gamma<\varepsilon\}}P\|^2
 &\leq C\int_0^{C\varepsilon}\rho^{\beta+n-2}
       \biggl\|\one_{[-C,C]}(t)(t^2+\rho^2)^{(1-n)/2}\biggr\|_{L^1(dt)}^2\,d\rho\notag\\
 &\quad+C\int_0^{C\varepsilon}\rho^{\beta+2-n}\,d\rho+C\varepsilon^{\beta+n-1}.
\end{align*}
The interior kernel norm satisfies
$$
 \int_{-C}^C(t^2+\rho^2)^{(1-n)/2}\,dt
 \leq C\left(\rho^{2-n}+\int_\rho^C t^{1-n}\,dt\right)
 \leq\begin{cases}C\rho^{2-n},&n\geq3,\\ C\log(e/\rho),&n=2.\end{cases}
$$
Consequently, for $n\geq3$,
\begin{align*}
 \|\one_{\{d_\Gamma<\varepsilon\}}P\|^2
 &\leq C\int_0^{C\varepsilon}\rho^{\beta+n-2}\rho^{4-2n}\,d\rho
       +C\int_0^{C\varepsilon}\rho^{\beta+2-n}\,d\rho+C\varepsilon^{\beta+n-1}\\
 &\leq C\int_0^{C\varepsilon}\rho^{\beta+2-n}\,d\rho+C\varepsilon^{\beta+n-1}\\
 &\leq C_\beta\varepsilon^{\beta+3-n}.
\end{align*}
For $n=2$, the substitution $\rho=\varepsilon u$ yields
\begin{align*}
 \|\one_{\{d_\Gamma<\varepsilon\}}P\|^2
 &\leq C\int_0^{C\varepsilon}\rho^\beta\log^2(e/\rho)\,d\rho+C\varepsilon^{\beta+1}\\
 &\leq C\varepsilon^{\beta+1}
       \int_0^C u^\beta\bigl(\log(e/\varepsilon)+|\log u|\bigr)^2\,du
       +C\varepsilon^{\beta+1}\\
 &\leq C_\beta\varepsilon^{\beta+1}\bigl(1+\log^2(e/\varepsilon)\bigr),
\end{align*}
since $\beta>-1$.

On the bounded tube, the operator $H_\varepsilon=\one_{U_{\rho_0}\cap\{d_\Gamma\geq\varepsilon\}}T$ is Hilbert--Schmidt, since
$$
 \|H_\varepsilon\|_{\HS}^2
 \leq c_{1,n}^2\varepsilon^{2-2n}\sigma(\Gamma)\omega_\beta(U_{\rho_0})<\infty.
$$
The preceding estimates also imply
$$
 \|T-H_\varepsilon\|\leq c_{1,n}\|\one_{\{d_\Gamma<\varepsilon\}}P\|\to0.
$$
The near-curve estimate and the bounded kernel away from $\Gamma$ show that $P$ is bounded on the tube. Since
$0\leq\chi_{\varepsilon,R}\leq1$, the smooth truncations are uniformly bounded there. For sufficiently large $R$,
the outer cutoff is one on $U_{\rho_0}\times\Gamma$, so
$$
 \|T_{\varepsilon,R,\sigma}-T\|\leq c_{1,n}\|\one_{\{d_\Gamma<\varepsilon\}}P\|\to0.
$$
Thus $T$ is the weak limit used in \cref{prop:green}, with target measure $\one_{U_{\rho_0}}\omega_\beta$.
Thus $T$ is compact, and so is $\mathcal G_\Omega=a_nT+S_\Omega$ by the Hilbert--Schmidt correction in \cref{prop:green}.

For $R>2+2\sup_{y\in\Gamma}|y|$ and $|x|>R$, the inequality $|x-y|\geq|x|/2$ implies
$$
 P|f|(x)\leq C|x|^{1-n}\int_\Gamma|f|\,ds_\Gamma
 \leq C|x|^{1-n}\sigma(\Gamma)^{1/2}\|f\|_2.
$$
For $\beta<n-2$, the comparison $d_\Gamma(x)\approx|x|$ implies
$$
 \|\one_{\{|x|>R\}}T\|^2
 \leq C\|\one_{\{|x|>R\}}P\|^2
 \leq C\int_R^\infty r^{\beta+1-n}\,dr
 =\frac{C}{n-2-\beta}R^{\beta+2-n}.
$$
The operators $H_{\varepsilon,R}=\one_{\{d_\Gamma\geq\varepsilon,\ |x|\leq R\}}T$ are Hilbert--Schmidt.
As $\varepsilon\to0^+$ and $R\to\infty$ independently,
$$
 \|T-H_{\varepsilon,R}\|^2
 \leq\|\one_{\{d_\Gamma<\varepsilon\}}T\|^2+\|\one_{\{|x|>R\}}T\|^2\to0.
$$
Thus $T$ is compact for $n-3<\beta<n-2$.

For necessity, suppose $n\geq3$. Choose an interior arclength parametrization $\gamma$ on $(-4r_*,4r_*)$,
with derivative modulus $\vartheta_\gamma$ and $\vartheta_\gamma(4r_*)\leq1/8$. For $0<r<r_*$ choose
$0\leq f_r\leq1$, equal to one on $\gamma([-2r,2r])$ and supported on $\gamma((-3r,3r))$.
For $|s|<r$ and $z\perp\gamma'(0)$ with $0<\rho:=|z|<r/4$, let $x=\gamma(s)+z$.
After decreasing $r_*$, these graph coordinates satisfy $d_\Gamma(x)\approx\rho$ and have Jacobian
$|\gamma'(s)\cdot\gamma'(0)|\in[7/8,1]$.

Let $\tau_s=\gamma'(s)$, $q_s=z\cdot\tau_s$ and $z_s=z-q_s\tau_s$. Then $|q_s|\leq\rho/8$ and
$|z_s|\approx\rho$. For $k(w)=c_{1,n}w|w|^{-n}$, define the tangent integral
$$
 L_s(z)=\int_{-r}^r k(t\tau_s-z)\,dt.
$$
Since $[q_s-|z_s|,q_s+|z_s|]\subset(-r,r)$,
\begin{align*}
 |L_s(z)|
 &\geq-\frac{z_s}{|z_s|}\cdot L_s(z)
 =c_{1,n}|z_s|\int_{-r}^r\bigl((t-q_s)^2+|z_s|^2\bigr)^{-n/2}\,dt
 \geq2^{1-n/2}c_{1,n}|z_s|^{2-n}\geq c_0\rho^{2-n},
\end{align*}
where  $c_0=2^{1-n/2}c_{1,n}.$

For $|t|<r$, write
$$
 \gamma(s+t)-\gamma(s)=t\tau_s+E_s(t),\qquad
 |E_s(t)|\leq|t|\vartheta_\gamma(|t|).
$$
The small slope implies $|t\tau_s-z+uE_s(t)|\geq c(t^2+\rho^2)^{1/2}$ for $0\leq u\leq1$.
Since $\|Dk(w)\|\leq C|w|^{-n}$, the mean value theorem followed by a dyadic decomposition yields
\begin{align*}
 \left|\int_{-r}^r k\bigl(\gamma(s+t)-\gamma(s)-z\bigr)\,dt-L_s(z)\right|
 &\leq\int_{-r}^r|E_s(t)|\int_0^1\|Dk(t\tau_s-z+uE_s(t))\|\,du\,dt\\
 &\leq C\int_0^r\frac{t\vartheta_\gamma(t)}{(t^2+\rho^2)^{n/2}}\,dt\\
 &\leq C\rho^{2-n}\bigg(\vartheta_\gamma(\rho)
   +\sum_{\substack{j\geq0\\2^j\rho<r}}2^{-j(n-2)}\vartheta_\gamma(2^{j+1}\rho)\bigg)\\
 &\leq C'\rho^{2-n}\vartheta_\gamma(2r).
\end{align*}
The last inequality follows from $\sum_{j\geq0}2^{-j(n-2)}<\infty$. On the remaining source,
$r\leq|t|\leq4r$ and $|\gamma(s+t)-\gamma(s)-z|\geq c|t|$, so
$$
 \left|Tf_r(\gamma(s)+z)-\int_{-r}^r k\bigl(\gamma(s+t)-\gamma(s)-z\bigr)\,dt\right|
 \leq C\int_{r\leq|t|\leq4r}|t|^{1-n}\,dt\leq Cr^{2-n}=:C_r.
$$
These constants are uniform in $s,z$, and $c_0,C'$ are independent of $r$. Consequently,
$$
 |Tf_r(\gamma(s)+z)|\geq\bigl(c_0-C'\vartheta_\gamma(2r)\bigr)\rho^{2-n}-C_r.
$$
Since $\gamma'$ is continuous, first fix $r$ with $C'\vartheta_\gamma(2r)\leq c_0/4$, and then choose
$\rho_1<\min\{r/4,\rho_0\}$ with $C_r\rho_1^{n-2}\leq c_0/4$. For the fixed function $f=f_r$,
$$
 |Tf(\gamma(s)+z)|\geq\tfrac12c_0\rho^{2-n},\qquad |s|<r,\quad0<\rho<\rho_1.
$$
The Green correction satisfies $|S_\Omega f|\leq\|B_\Omega\|_\infty\|f\|_1$, so after decreasing $\rho_1$,
$|\mathcal G_\Omega f(\gamma(s)+z)|\gtrsim\rho^{2-n}$ as well. Hence
$$
 \int_{U_{\rho_0}}|Tf(x)|^2d_\Gamma(x)^\beta\,dx
 \geq c\int_0^{\rho_1}\rho^{4-2n}\rho^{\beta+n-2}\,d\rho
 =c\int_0^{\rho_1}\rho^{\beta+2-n}\,d\rho=\infty
 \qquad(\beta\leq n-3).
$$
For $n=2$, local finiteness already requires $\beta>-1=n-3$. At infinity, the mean value theorem yields
$$
 \left|T\one_\Gamma(x)+c_{1,n}\sigma(\Gamma)\frac{x}{|x|^n}\right|
 \leq\int_\Gamma\left|\bfK^{1,n}(x,y)+c_{1,n}\frac{x}{|x|^n}\right|\,ds_\Gamma(y)
 \leq C|x|^{-n}\int_\Gamma|y|\,ds_\Gamma(y).
$$
The leading term has size $c_{1,n}\sigma(\Gamma)|x|^{1-n}$, so
$$
 \int_{|x|>R}|T\one_\Gamma(x)|^2d_\Gamma(x)^\beta\,dx
 \geq c\int_R^\infty r^{2-2n}r^{\beta+n-1}\,dr
 =c\int_R^\infty r^{\beta+1-n}\,dr=\infty
 \qquad(\beta\geq n-2)
$$
for sufficiently large $R$. The divergences at $\beta=n-3$ ($n\geq3$) and $\beta=n-2$ are logarithmic.
\end{proof}

\begin{remark}\label{rem:distance-trunc}
In the compactness ranges above, the smooth truncations from Section~\ref{subsec:truncations} converge in norm
to the ordinary integral $T$. The inner error is supported where $|x-y|<\varepsilon$, hence $d_\Gamma(x)<\varepsilon$.
For large $R$, the outer error is supported where $|x-y|>R$, hence $|x|>R/2$. Thus, in either target space,
$$
 \|T-T_{\varepsilon,R,\sigma}\|
 \leq C\bigl(\|\one_{\{d_\Gamma<\varepsilon\}}P\|+\|\one_{\{|x|>R/2\}}P\|\bigr)\to0.
$$
The estimates above show that $P$ is bounded in either target space. Hence the smooth truncations are uniformly bounded,
and the ordinary integral $T$ is their weak operator limit.
On the bounded tube the outer term is zero for large $R$. Any uniformly bounded smooth cutoff family that equals
one for $C\varepsilon\leq|x-y|\leq R/C$, with fixed $C\geq1$, satisfies the same estimate with fixed multiples of the radii.
\end{remark}

\subsection{The Hilbert--Schmidt condition}

\begin{proposition}\label{prop:distance-hilbert-schmidt}
For the bounded tube in \cref{thm:compact-distance-powers}, the Riesz operator is Hilbert--Schmidt if and
only if $\beta>n-2$. The Green gradient is Hilbert--Schmidt if and only if $\beta>n-2$.  In the whole-space range $n-3<\beta<n-2$,
the Riesz operator is compact but is not Hilbert--Schmidt.
\end{proposition}

\begin{proof}
In a smaller interior graph neighbourhood, with $\rho=d_\Gamma(x)$, the squared kernel satisfies
$$
 \int_\Gamma|\bfK^{1,n}(x,y)|^2\,ds_\Gamma(y)
 \approx\int_{-c}^c(t^2+\rho^2)^{1-n}\,dt\approx\rho^{3-2n}.
$$
Restricting the integral to $|t|<\rho$ proves the lower bound. The substitution $t=\rho u$ in the integral over $\R$
proves the upper bound. The interior contribution is therefore finite exactly when
$$
 \int_0^{\rho_0}\rho^{\beta+n-2}\rho^{3-2n}\,d\rho
 =\int_0^{\rho_0}\rho^{\beta+1-n}\,d\rho<\infty,
 \quad\text{that is,}\quad\beta>n-2.
$$
Beyond an endpoint, with $\delta=(\dist(s,J)^2+|z|^2)^{1/2}$, we have
$$
 \int_\Gamma|\bfK^{1,n}(x,y)|^2\,ds_\Gamma(y)
 \leq C+C\int_0^\infty(t^2+\delta^2)^{1-n}\,dt\leq C\delta^{3-2n}.
$$
Its weighted integral is at most
$$
 C\int_0^{\rho_0}\delta^{\beta+n-1}\delta^{3-2n}\,d\delta
 =C\int_0^{\rho_0}\delta^{\beta+2-n}\,d\delta<\infty
 \qquad(\beta>n-3).
$$
This endpoint condition is weaker than the interior requirement $\beta>n-2$.
The remaining kernel is bounded on a product of finite measure spaces, so $T$ is Hilbert--Schmidt for $\beta>n-2$.

Conversely, if $T$ were Hilbert--Schmidt, restriction to a fixed smaller interior graph neighbourhood $V$ would imply
\begin{align*}
 \|T\|_{\HS}^2
 &\geq\|\one_{V\cap\{d_\Gamma>\varepsilon\}}T\|_{\HS}^2
 =\int_{V\cap\{d_\Gamma>\varepsilon\}}\int_\Gamma
       |\bfK^{1,n}(x,y)|^2\,ds_\Gamma(y)\,d\omega_\beta(x)\\
 &\geq c\int_{C\varepsilon}^{\rho_1}\rho^{\beta+1-n}\,d\rho
 \to\infty\qquad(\beta\leq n-2).
\end{align*}
Each restricted kernel is square integrable, so its squared Hilbert--Schmidt norm is the double integral above.
The assertion for the Green gradient follows from the Hilbert--Schmidt correction in \cref{prop:green}.

At infinity,
$$
 \int_{|x|>R}\int_\Gamma|\bfK^{1,n}(x,y)|^2\,ds_\Gamma(y)\,d\omega_\beta(x)
 \approx\int_R^\infty r^{\beta+1-n}\,dr<\infty
 \quad\Longleftrightarrow\quad\beta<n-2.
$$
This is incompatible with the local requirement $\beta>n-2$.
\end{proof}

\subsection{Biot--Savart operators}

Let $\Gamma\subset\R^3$ be a compact regular $C^1$ curve, let $\sigma$ be a finite positive measure supported on
$\Gamma$, and let $\omega$ be locally finite. For $\bfK^{1,3}$, assume that $\mathfrak N_s(\sigma,\omega)<\infty$,
and let $T$ be the weak operator limit of the smooth truncations.

Let $e_1,e_2,e_3$ be the standard basis of $\mathbb R^3$.
For $F\in L^2(\sigma;\mathbb C^3)$, write
$F=\sum_{k=1}^3 F_ke_k$, where $F_k\in L^2(\sigma;\mathbb C)$
is the $k$th component of $F$. 
Let $A_kv=e_k\times v$, $P_kF=F_k$, $J_kf=fe_k$ and $c_B=-(4\pi c_{1,3})^{-1}$, with $A_k$ acting complex linearly.

Using the component identity from \cite[proof of Corollary~8.4]{BoundedDini}, define
$$
 \mathcal B_\sigma=c_B\sum_{k=1}^3A_kTP_k,
 \qquad \mathcal B_\sigma J_k=c_BA_kT.
$$
Off $\Gamma$, $\mathcal B_\sigma$ agrees $\omega$-almost everywhere with the absolutely convergent integral
$$
 \mathcal B_\sigma F(x)=\frac1{4\pi}\int_\Gamma
       \frac{F(y)\times(x-y)}{|x-y|^3}\,d\sigma(y),\qquad F\in L^2(\sigma;\mathbb C^3).
$$

\begin{proposition}\label{prop:compact-biot-savart}
The operator $\mathcal B_\sigma:L^2(\sigma;\mathbb C^3)\to L^2(\omega;\mathbb C^3)$ is compact if and only if $T:L^2(\sigma)\to
L^2(\omega;\mathbb C^3)$ is compact.  Moreover,
$$
 \frac{2|c_B|}{3}\|T\|_{\mathrm{ess}}
 \leq\|\mathcal B_\sigma\|_{\mathrm{ess}}
 \leq\sqrt3\,|c_B|\|T\|_{\mathrm{ess}}.
$$
When $\Gamma$ is $C^{1,\mathrm{Dini}}$, \cref{thm:main} therefore characterizes compactness of the Biot--Savart operator
for arbitrary vector densities.
\end{proposition}

\begin{proof}
Since $\sum_{k=1}^3A_k^2=-2I$, the identity $\mathcal B_\sigma J_k=c_BA_kT$ implies
$$
 T=-\frac1{2c_B}\sum_{k=1}^3A_k\mathcal B_\sigma J_k.
$$
For a compact operator $K$ with the domain and range of $T$, the operator $K_B=c_B\sum_kA_kKP_k$ is compact, and
\begin{align*}
 \|\mathcal B_\sigma\|_{\ess}
 &\leq\|\mathcal B_\sigma-K_B\|
 =|c_B|\sup_{\|F\|_2\leq1}\left\|\sum_{k=1}^3A_k(T-K)F_k\right\|_2\\
 &\leq\sqrt3\,|c_B|\sup_{\|F\|_2\leq1}
       \left(\sum_{k=1}^3\|(T-K)F_k\|_2^2\right)^{1/2}\\
 &\leq\sqrt3\,|c_B|\|T-K\|.
\end{align*}
For a compact operator $L$ with the domain and range of $\mathcal B_\sigma$, let $L_T=-(2c_B)^{-1}\sum_kA_kLJ_k$.
Then $L_T$ is compact, and
\begin{align*}
 \|T\|_{\ess}
 &\leq\|T-L_T\|
 =\frac1{2|c_B|}\left\|\sum_{k=1}^3A_k(\mathcal B_\sigma-L)J_k\right\|\\
 &\leq\frac1{2|c_B|}\sum_{k=1}^3\|A_k\|\|\mathcal B_\sigma-L\|\|J_k\|\\
 &\leq\frac3{2|c_B|}\|\mathcal B_\sigma-L\|.
\end{align*}
Taking infima over $K$ and $L$ proves the essential-norm estimates and hence the compactness equivalence.
The two identities hold pointwise for the ordinary integrals off $\Gamma$. When $\omega(\Gamma)=0$, boundedness of
either ordinary integral therefore implies boundedness of the other.
\end{proof}

Under the assumptions of \cref{thm:compact-distance-powers} with $n=3$, let $\sigma=ds_\Gamma$ and
$d\omega_\beta(x)=d_\Gamma(x)^\beta\,dx$. Then \cref{thm:compact-distance-powers,prop:compact-biot-savart} give
$$
 \mathcal B_{ds_\Gamma}\text{ is compact on a bounded tube}
 \quad\Longleftrightarrow\quad\beta>0,
$$
and
$$
 \mathcal B_{ds_\Gamma}\text{ is compact on }\R^3
 \quad\Longleftrightarrow\quad0<\beta<1.
$$
For necessity, suppose that $\mathcal B_{ds_\Gamma}$ is bounded on the tube or on $\R^3$. Since $\omega_\beta(\Gamma)=0$,
the identities in the proof of \cref{prop:compact-biot-savart} imply that $T$ is bounded on the same set. The necessity
proof in \cref{thm:compact-distance-powers} then gives the stated ranges.
The sum of the squared matrix entries of $v\mapsto v\times z$ is $2|z|^2$, so the squared matrix kernel norm is
$2(4\pi)^{-2}|x-y|^{-4}$. By \cref{prop:distance-hilbert-schmidt}, the bounded-tube operator is Hilbert--Schmidt
exactly when $\beta>1$, and none of the whole-space operators with $0<\beta<1$ is Hilbert--Schmidt.

Let $\tau$ be the unit tangent along $\Gamma$. Since multiplication $f\mapsto f\tau$ is bounded from $L^2(ds_\Gamma)$ to
$L^2(ds_\Gamma;\mathbb C^3)$, the field operator $f\mapsto\mathcal B_{ds_\Gamma}(f\tau)$ for a current along
$\Gamma$ is compact in the stated ranges.

\begin{corollary}\label{cor:weak-source-fields}
Let $n\geq2$, and fix the bounded $C^2$ domain $\Omega$, the compact regular simple $C^1$ curve
$\Gamma\Subset\Omega$, and the tube $\overline U_{\rho_0}\subset\Omega$ as above.
Suppose that $f_j\rightharpoonup f$ in $L^2(ds_\Gamma)$. Define
$$
 u_j(x)=\int_\Gamma G_\Omega(x,y)f_j(y)\,ds_\Gamma(y),
 \qquad
 u(x)=\int_\Gamma G_\Omega(x,y)f(y)\,ds_\Gamma(y).
$$
If $\beta>n-3$, then
$$
 \int_{U_{\rho_0}}
 |\nabla u_j-\nabla u|^2d_\Gamma(x)^\beta\,dx\to0.
$$
In particular, the weighted field energies converge:
$$
 \int_{U_{\rho_0}}|\nabla u_j|^2d_\Gamma(x)^\beta\,dx\to\int_{U_{\rho_0}}|\nabla u|^2d_\Gamma(x)^\beta\,dx.
$$

In dimension three, suppose instead that $F_j\rightharpoonup F$ in $L^2(ds_\Gamma;\mathbb C^3)$. If $0<\beta<1$, then
$$
 \|\mathcal B_{ds_\Gamma}F_j-\mathcal B_{ds_\Gamma}F\|_{L^2(\R^3,d_\Gamma(x)^\beta\,dx;\mathbb C^3)}\to0.
$$
\end{corollary}

\begin{proof}
By \cref{thm:compact-distance-powers}, $\mathcal G_\Omega$ is compact. Thus $f_j\rightharpoonup f$ implies
$$
 \|\nabla u_j-\nabla u\|_{L^2(U_{\rho_0},\omega_\beta)}
 =\|\mathcal G_\Omega(f_j-f)\|_{L^2(U_{\rho_0},\omega_\beta)}\to0.
$$
The energy difference satisfies
$$
 \bigl|\|\nabla u_j\|_2^2-\|\nabla u\|_2^2\bigr|
 \leq(\|\nabla u_j\|_2+\|\nabla u\|_2)\|\nabla u_j-\nabla u\|_2\to0.
$$
For Biot--Savart fields, \cref{thm:compact-distance-powers,rem:distance-trunc} show that the whole-space Riesz integral
is compact. By \cref{prop:compact-biot-savart},
$\mathcal B_{ds_\Gamma}$ is compact, so weak convergence of the densities again implies strong convergence of the fields.
\end{proof}

\section{Tangent operators and an exact essential norm}\label{sec:tangent}

We compare the curve kernel with the tangential Hilbert kernel to prove \cref{thm:tangent-exact}.

Assume $n\geq2$.  Let $\mathbb T_L=\R/L\mathbb Z$, with Lebesgue measure $ds$ and cyclic distance
$\delta_L(s,t)=\min_{k\in\mathbb Z}|s-t+kL|$. Let $\Phi:\mathbb T_L\to\R^n$ be a regular $C^{1,\mathrm{Dini}}$
embedding parametrized by arclength, and let $\Gamma=\Phi(\mathbb T_L)$ and $\tau(s)=\Phi'(s)$. Thus $|\tau|=1$.
Define
$$
 \vartheta(r)=\sup_{\delta_L(s,t)\leq r}|\tau(s)-\tau(t)|,
 \qquad \int_0^{L/2}\vartheta(r)\,\frac{dr}{r}<\infty.
$$
For the positive densities $a,b\in C(\Gamma)$, use the same letters for their compositions with $\Phi$. Then
$$
 \sigma=\Phi_\#(a(s)\,ds),\qquad \omega=\Phi_\#(b(s)\,ds).
$$
Since $a,b$ are positive and continuous, choose $m,M>0$ with $m\leq a(s),b(s)\leq M$.
Let $c=c_{n-1,n}$ and use $T^\sharp_\varepsilon$ from the Introduction. The maps
$$
 U_\sigma f(s)=\sqrt{a(s)}f(\Phi(s)),\qquad
 U_\omega F(s)=\sqrt{b(s)}F(\Phi(s)),
$$
are unitary onto $L^2(ds)$ and $L^2(ds;\mathbb C^n)$, respectively.
Let $M_v$ denote multiplication by $v$, and define
$$
 \mathcal J_Lg(s)=\operatorname{p.v.}\int_{\mathbb T_L}
       \frac{\pi}{L}\cot\frac{\pi(t-s)}{L}\,g(t)\,dt.
$$

\subsection{The Hilbert transform and the curve kernel}

For $k\in\mathbb Z$, let $e_k(s)=L^{-1/2}e^{2\pi iks/L}$. The kernel $k_L(h)=(\pi/L)\cot(\pi h/L)$ is odd and
$L$-periodic, with $k_L(h)=h^{-1}+O_L(|h|)$ near zero.
Let $\mathcal J_{L,\varepsilon}$ denote its truncation to $\delta_L(s,t)>\varepsilon$.
The compactness of commutators with continuous multipliers is classical, see \cite{Hartman}.

\begin{lemma}\label{lem:tangent-hilbert}
The multiplier of $\mathcal J_L$ is $i\pi\operatorname{sgn}(k)$, with $\operatorname{sgn}(0)=0$, and
$\|\mathcal J_L\|=\pi$. Its symmetric arclength truncations are uniformly bounded and converge strongly on
$L^2(\mathbb T_L)$. The principal value exists almost everywhere. For $q\in C(\mathbb T_L)$,
$[\mathcal J_L,M_q]$ is compact. If $v\in C(\mathbb T_L;\mathbb C^\ell)$, $\ell\geq1$, then
$$
 \|M_v\mathcal J_L\|_{\ess}=\pi\|v\|_\infty.
$$
\end{lemma}

\begin{proof}
The multiplier is
$$
 m(k)=\operatorname{p.v.}\int_{-L/2}^{L/2}k_L(h)e^{2\pi ikh/L}\,dh
 =i\int_0^\pi\cot(\theta/2)\sin(k\theta)\,d\theta=i\pi\operatorname{sgn}(k).
$$
For $k\geq1$, the last equality follows by integrating
$\cot(\theta/2)\sin(k\theta)=1+\cos(k\theta)+2\sum_{j=1}^{k-1}\cos(j\theta)$. The other frequencies follow by oddness.

Since $k_L(h)-h^{-1}\in L^1(0,L/2)$, the truncated multipliers satisfy, for $k\ne0$ and $0<\varepsilon<L/2$,
$$
 |m_\varepsilon(k)|
 \leq2\bigg|\int_{2\pi|k|\varepsilon/L}^{\pi|k|}\frac{\sin u}{u}\,du\bigg|+C_L\leq C_L'.
$$
The sine integrals are uniformly bounded by integration by parts above $1$. The multipliers vanish when $k=0$ or
$\varepsilon\geq L/2$. Parseval's identity and dominated convergence prove strong convergence, and
$|m(k)|=\pi$ for $k\ne0$ yields $\|\mathcal J_L\|=\pi$.
On compact subintervals of arcs shorter than $L/2$, the principal value reduces to the real-line Hilbert transform
plus an integrable remainder. Almost-everywhere convergence follows from \cite[Chapter~II]{Stein} and a finite covering.

For $q(s)=e^{2\pi ims/L}$,
$$
 [\mathcal J_L,M_q]e_k
 =i\pi\bigl(\operatorname{sgn}(k+m)-\operatorname{sgn}(k)\bigr)e_{k+m},
$$
which vanishes outside finitely many $k$. If trigonometric polynomials $q_j$ converge uniformly to $q$,
$$
 \|[\mathcal J_L,M_q]-[\mathcal J_L,M_{q_j}]\|
 \leq2\pi\|q-q_j\|_\infty\to0,
$$
so the commutator is compact.

Let $A=M_v\mathcal J_L$ and $V=\|v\|_\infty$. We have $\|A\|_{\ess}\leq\|A\|\leq\pi V$.
For $V>0$ and $0<\eta<V$, choose $\psi\in C^\infty(\mathbb T_L)$ of norm one, supported in an arc where $|v|>V-\eta$.
The functions $f_N(s)=\psi(s)e^{2\pi iNs/L}$ tend weakly to zero, since their pairings with $g\in L^2$ are Fourier
coefficients of $\psi\overline g\in L^1$. If $\psi=\sum_j\psi_je_j$, then
$$
 \|\mathcal J_Lf_N-i\pi f_N\|_2^2
 =\pi^2|\psi_{-N}|^2+4\pi^2\sum_{j<-N}|\psi_j|^2\to0.
$$
For every compact $C$, therefore,
$$
 \|A-C\|\geq\lim_{N\to\infty}\|(A-C)f_N\|_2
 =\pi\|v\psi\|_2\geq\pi(V-\eta).
$$
Taking the infimum over $C$ and letting $\eta\to0^+$ proves the lower bound. The case $V=0$ is trivial.
\end{proof}

For $s\ne t$ on the parameter circle, define
$$
 G(s,t)=\frac{\Phi(t)-\Phi(s)}{|\Phi(t)-\Phi(s)|^2},
 \qquad E(s,t)=G(s,t)-\tau(s)k_L(t-s).
$$

\begin{lemma}\label{lem:tangent-remainder}
There exists $r_0\in(0,L/4)$ such that, for $0<|h|<r_0$,
$$
 (1-\vartheta(|h|))|h|\leq|\Phi(s+h)-\Phi(s)|\leq|h|,
 \qquad \vartheta(r_0)\leq\tfrac12.
$$
The difference kernel satisfies
$$
 |E(s,t)|\leq C\left(\frac{\vartheta(\delta_L(s,t))}{\delta_L(s,t)}+\delta_L(s,t)\right),
 \qquad 0<\delta_L(s,t)<r_0,
$$
and is bounded on every set $\delta_L(s,t)\geq r>0$. The operator
$\mathcal Eg(s)=\int E(s,t)g(t)\,dt$ is compact from $L^2(\mathbb T_L)$ to $L^2(\mathbb T_L;\mathbb C^n)$,
and the integral converges absolutely almost everywhere for each $g\in L^2$.
\end{lemma}

\begin{proof}
For $0<|h|<L/2$, using oriented integrals when $h<0$, define
$$
 e(s,h)=\frac1h\int_0^h(\tau(s+u)-\tau(s))\,du.
$$
Then $|e(s,h)|\leq\vartheta(|h|)$ and $\Phi(s+h)-\Phi(s)=h(\tau(s)+e(s,h))$.
Choose $r_0<L/4$ with $\vartheta(r_0)\leq1/2$. The chord bounds follow from arclength parametrization and
$|e(s,h)|\leq\vartheta(|h|)$. For $0<|h|<r_0$, abbreviate $e=e(s,h)$. Then
\begin{align*}
 \left|G(s,s+h)-\frac{\tau(s)}h\right|
 &=\frac{|e-(2\tau(s)\cdot e+|e|^2)\tau(s)|}{|h|\,|\tau(s)+e|^2}
 \leq\frac{3|e|+|e|^2}{|h|(1-|e|)^2}\leq C\frac{\vartheta(|h|)}{|h|}.
\end{align*}
The expansion $k_L(h)=h^{-1}+O_L(|h|)$ proves the bound for $E$.
For $r>0$, injectivity and compactness imply
$$
 \inf_{\delta_L(s,t)\geq r}|\Phi(t)-\Phi(s)|>0,
$$
so $E$ is bounded off the parameter diagonal.

Let $\mathcal E_r$ have kernel $\one_{\{\delta_L(s,t)\geq r\}}E(s,t)$. It is Hilbert--Schmidt, and the Schur test yields
\begin{align}\label{eq:tangent-remainder-rate}
 \|\mathcal E-\mathcal E_r\|
 &\leq\bigg(\sup_s\int_{\delta_L(s,t)<r}|E(s,t)|\,dt\,
                  \sup_t\int_{\delta_L(s,t)<r}|E(s,t)|\,ds\bigg)^{1/2}\notag\\
 &\leq C\int_0^r\left(\frac{\vartheta(u)}u+u\right)du\to0.
\end{align}
Thus $\mathcal E$ is compact. Since $L^2(\mathbb T_L)\subset L^1(\mathbb T_L)$, Tonelli also yields
$$
 \int_{\mathbb T_L}\int_{\mathbb T_L}|E(s,t)|\,|g(t)|\,dt\,ds
 \leq\bigg(\sup_t\int_{\mathbb T_L}|E(s,t)|\,ds\bigg)\|g\|_1<\infty.
$$
The proof is complete.
\end{proof}

Define $S=M_\tau\mathcal J_L+\mathcal E$. The arclength truncations of $G$ satisfy
$S^{\mathrm{arc}}_\varepsilon=M_\tau\mathcal J_{L,\varepsilon}+\mathcal E_\varepsilon$, so
\cref{lem:tangent-hilbert} and \eqref{eq:tangent-remainder-rate} imply $S^{\mathrm{arc}}_\varepsilon g\to Sg$ in $L^2$.
The same limit holds almost everywhere by the Hilbert principal-value theorem and absolute convergence of the remainder.

\subsection{Euclidean and smooth truncations}

Let $S^\sharp_\varepsilon$ be the truncation of $G$ to $|\Phi(t)-\Phi(s)|>\varepsilon$.
We compare it with $S^{\mathrm{arc}}_\varepsilon$, truncated to $\delta_L(s,t)>\varepsilon$.

\begin{lemma}\label{lem:tangent-truncations}
The difference $S^\sharp_\varepsilon-S^{\mathrm{arc}}_\varepsilon$ tends to zero in operator norm and, on each $g\in
L^2$, pointwise almost everywhere.  The operators $S^\sharp_\varepsilon$, $\varepsilon>0$, are uniformly bounded.
The operators $U_\omega T_{\varepsilon,R,\sigma}U_\sigma^{-1}$ are uniformly bounded in both parameters and
converge strongly to $cM_{\sqrt b}SM_{\sqrt a}$.
\end{lemma}

\begin{proof}
Choose
$$
 0<\varepsilon_0<\min\left\{\frac{r_0}{2},
                \inf_{\delta_L(s,t)\geq r_0}|\Phi(t)-\Phi(s)|\right\}.
$$
For $\varepsilon<\varepsilon_0$, the difference set
$A_\varepsilon(s)=\{t:\delta_L(s,t)>\varepsilon\geq|\Phi(t)-\Phi(s)|\}$ satisfies
$$
 \varepsilon<\delta_L(s,t)<r_0,\qquad
 (1-\vartheta(\delta_L(s,t)))\delta_L(s,t)\leq\varepsilon.
$$
Since $\vartheta(r_0)\leq1/2$, these inequalities imply $\delta_L(s,t)\leq2\varepsilon$ and hence
$$
 A_\varepsilon(s)\subset
 \left\{t:\varepsilon<\delta_L(s,t)\leq\frac{\varepsilon}{1-\vartheta(2\varepsilon)}\right\}.
$$
Since $|G(s,t)|\leq C/\delta_L(s,t)$ there, the two Schur bounds imply
$$
 \|S^\sharp_\varepsilon-S^{\mathrm{arc}}_\varepsilon\|
 \leq C\int_\varepsilon^{\varepsilon/(1-\vartheta(2\varepsilon))}\frac{du}{u}
 =C\log\frac1{1-\vartheta(2\varepsilon)}\to0.
$$
For every Lebesgue point $s$ of $g$,
\begin{align*}
 |(S^\sharp_\varepsilon-S^{\mathrm{arc}}_\varepsilon)g(s)|
 &=\left|\int_{A_\varepsilon(s)}G(s,t)g(t)\,dt\right|\\
 &\leq C|g(s)|\log\frac1{1-\vartheta(2\varepsilon)}
     +\frac C\varepsilon\int_{|h|<2\varepsilon}|g(s+h)-g(s)|\,dh\to0.
\end{align*}
For $\varepsilon\geq\varepsilon_0$, $|G|\one_{\{|\Phi(t)-\Phi(s)|>\varepsilon\}}\leq1/\varepsilon_0$.
For $0<\varepsilon<\varepsilon_0$, the arclength truncations are uniformly bounded by
\cref{lem:tangent-hilbert,lem:tangent-remainder}. The estimate above then gives a uniform bound for $S^\sharp_\varepsilon$.
Thus $\sup_{\varepsilon>0}\|S^\sharp_\varepsilon\|<\infty$.
The norm comparison with $S^{\mathrm{arc}}_\varepsilon$ proves $S^\sharp_\varepsilon g\to Sg$ in $L^2$.

For $\rho>0$ and $\varepsilon<R$, at least one cutoff factor equals one. By the fundamental theorem of calculus,
\begin{align*}
 \eta_0(\rho/\varepsilon)\eta_0(R/\rho)
 &=\eta_0(\rho/\varepsilon)+\eta_0(R/\rho)-1
 =\int_{1/2}^1\eta_0'(u)
       \bigl(\one_{\{\rho>\varepsilon u\}}-\one_{\{\rho>R/u\}}\bigr)\,du.
\end{align*}
For fixed $0<\varepsilon<R$, the kernels of $T^\sharp_{\varepsilon u}$ and $T^\sharp_{R/u}$ are jointly measurable
and uniformly bounded for $1/2\leq u\leq1$. Since $\sigma$ and $\omega$ are finite, the two families are strongly
measurable in the Hilbert--Schmidt norm. Integrating against the kernel and applying Fubini to bounded functions, we obtain
$$
 T_{\varepsilon,R,\sigma}
 =\int_{1/2}^1\eta_0'(u)\bigl(T^\sharp_{\varepsilon u}-T^\sharp_{R/u}\bigr)\,du.
$$
Since $U_\omega T^\sharp_\rho U_\sigma^{-1}=cM_{\sqrt b}S^\sharp_\rho M_{\sqrt a}$, the integrals extend strongly to $L^2$ and
$$
 \sup_{0<\varepsilon<R}\|T_{\varepsilon,R,\sigma}\|
 \leq2\bigg(\int_{1/2}^1|\eta_0'(u)|\,du\bigg)\sup_{\rho>0}\|T^\sharp_\rho\|<\infty.
$$
For $R\geq\diam\Gamma$, the outer term vanishes. Since $\int_{1/2}^1\eta_0'=1$, for $g\in L^2$ we have
\begin{align*}
 &\|U_\omega T_{\varepsilon,R,\sigma}U_\sigma^{-1}g-cM_{\sqrt b}SM_{\sqrt a}g\|_2\leq c\|\sqrt b\|_\infty\int_{1/2}^1|\eta_0'(u)|\,
             \|(S^\sharp_{\varepsilon u}-S)(\sqrt a\,g)\|_2\,du\to0
\end{align*}
by dominated convergence. Thus $U_\omega TU_\sigma^{-1}=cM_{\sqrt b}SM_{\sqrt a}$.
\end{proof}

\begin{proof}[Proof of \cref{thm:tangent-exact}]
By \cref{lem:tangent-truncations},
\begin{align}\label{eq:tangent-decomposition}
 U_\omega TU_\sigma^{-1}
 &=cM_{\sqrt b}(M_\tau\mathcal J_L+\mathcal E)M_{\sqrt a}\notag\\
 &=cM_{\sqrt{ab}\,\tau}\mathcal J_L
     +cM_{\sqrt b\,\tau}[\mathcal J_L,M_{\sqrt a}]+cM_{\sqrt b}\mathcal E M_{\sqrt a}\notag\\
 &=cM_{\sqrt{ab}\,\tau}\mathcal J_L+K.
\end{align}
By \cref{lem:tangent-remainder}, $\mathcal E$ is compact. 
The finite-rank approximation in \cref{lem:tangent-hilbert} proves compactness of
$[\mathcal J_L,M_{\sqrt a}]$. Thus $K:L^2(\mathbb T_L)\to L^2(\mathbb T_L;\mathbb C^n)$ is compact.
By the same lemma and $|\tau|=1$,
$$
 \|T\|_{\ess}=c\|M_{\sqrt{ab}\,\tau}\mathcal J_L\|_{\ess}
 =\pi c\|\sqrt{ab}\,\tau\|_\infty=\pi c\max_{\mathbb T_L}\sqrt{ab}.
$$
For the pointwise assertion, $g(t)=a(t)f(\Phi(t))$ satisfies
$\|g\|_2^2\leq M\|f\|_{L^2(\sigma)}^2$. Its arclength and Euclidean principal values agree by
\cref{lem:tangent-truncations}, so this principal value represents $T$.
\end{proof}

\subsection{Tangent models and projections}

At $s_0\in\mathbb T_L$, differentiability of $\Phi$ and continuity of $a,b$ imply
$$
 c\sqrt{b(s_0+ru)}\,rG(s_0+ru,s_0+rv)\sqrt{a(s_0+rv)}
 \to c\sqrt{a(s_0)b(s_0)}\frac{\tau(s_0)}{v-u},
$$
locally uniformly away from $u=v$, as $r\to0^+$. The corresponding line operator is
$$
 T^{\mathrm{tan}}_{s_0}=c\sqrt{a(s_0)b(s_0)}\,\tau(s_0)\mathcal J_{\R},
 \qquad \mathcal J_{\R}g(u)=\operatorname{p.v.}\int_{\R}\frac{g(v)}{v-u}\,dv.
$$
With $\widehat g(\xi)=\int e^{-iu\xi}g(u)\,du$, the multiplier of $\mathcal J_{\R}$ is
$i\pi\operatorname{sgn}\xi$. Thus $\|T^{\mathrm{tan}}_{s_0}\|=\pi c\sqrt{a(s_0)b(s_0)}$,
and \cref{thm:geometric,thm:tangent-exact} imply the following identity.

\begin{corollary}\label{cor:tangent-local}
For the present pair,
$
 m_T(0)=\|T\|_{\ess}
 =\max_{s_0\in\mathbb T_L}\|T^{\mathrm{tan}}_{s_0}\|
 =\pi c\max_{\mathbb T_L}\sqrt{ab}.
$
\end{corollary}

For constant densities, the essential norm is independent of the length and shape of the closed curve.
The uniform bound for the smooth truncations and \cref{prop:bounded-smooth} imply $\Theta(\sigma,\omega)<\infty$. Hence
\cref{thm:main} yields $\mathfrak V_0\approx\pi c\max_{\mathbb T_L}\sqrt{ab}>0$.

\begin{corollary}\label{cor:tangent-projections}
Let $q\in C(\mathbb T_L;\R^n)$, and define $T_q$ by $T_qf(\Phi(s))=q(s)\cdot Tf(\Phi(s))$.  Then
$$
 \|T_q\|_{\ess}
   =\pi c\max_{s\in\mathbb T_L}\sqrt{a(s)b(s)}\,|q(s)\cdot\tau(s)|.
$$
In particular, every continuous normal component, for which $q\cdot\tau=0$, is compact.  The tangential projection
$q=\tau$ has the same essential norm as the full vector.
\end{corollary}

\begin{proof}
Multiplying \eqref{eq:tangent-decomposition} by $F(s)\mapsto q(s)\cdot F(s)$ leaves a compact remainder.
Applying \cref{lem:tangent-hilbert} with $\ell=1$, we obtain
$$
 \|T_q\|_{\ess}
 =c\|M_{\sqrt{ab}(q\cdot\tau)}\mathcal J_L\|_{\ess}
 =\pi c\max_{\mathbb T_L}\sqrt{ab}\,|q\cdot\tau|.
$$
The proof is complete.
\end{proof}

For $n=2$, take $q=\nu\circ\Phi$ for a continuous unit normal $\nu$ to $\Gamma$, and denote the projection by $T_\nu$.
The classical double-layer boundary operator is
$$
 Df(x)=\frac1{2\pi}\operatorname{p.v.}\int_\Gamma
       \frac{\nu(y)\cdot(y-x)}{|x-y|^2}f(y)\,ds_\Gamma(y).
$$
With this normalization, $T_\nu f=-2\pi cD^*(af)$, where $D^*$ is the adjoint on $L^2(ds_\Gamma)$.
The normal component is therefore a weighted adjoint double-layer operator. Its compactness on $C^1$ boundaries,
as in \cite[Section~3D]{HMT}, is preserved by the bounded density multipliers.

\begin{corollary}\label{cor:tangent-orders}
For the closed curve and positive continuous densities above, let $0\leq\alpha<n$. Then:
\begin{enumerate}
\item If $\alpha>n-1$, the integral operator is compact and the smooth truncations converge in operator norm.
\item If $\alpha=n-1$, the operator is bounded and noncompact, with essential norm given by \eqref{eq:tangent-exact}.
\item If $\alpha<n-1$, then $\mathfrak N_s(\sigma,\omega)=\infty$. There is no bounded operator agreeing with the
kernel on separated supports.
\end{enumerate}
\end{corollary}

\begin{proof}
Let $d=n-\alpha$. By \cref{lem:tangent-remainder}, chord and arclength distances are comparable near the diagonal,
and the kernel is bounded on every set separated from the diagonal. If $0<d<1$, the Schur bound for the strip $\delta_L(s,t)<r$ is at most
$$
 C\int_0^r u^{-d}\,du=\frac{Cr^{1-d}}{1-d}\to0.
$$
On the complementary region the kernel is bounded and both measures are finite, so the operator is Hilbert--Schmidt.
The strip estimate therefore proves compactness of the integral operator.
For $R\geq\diam\Gamma$ and small $\varepsilon$, its difference from $T_{\varepsilon,R,\sigma}$ is supported where
$|\Phi(t)-\Phi(s)|<\varepsilon$, hence $\delta_L(s,t)<2\varepsilon$, and has norm at most $C\varepsilon^{1-d}$.
The case $d=1$ is \cref{thm:tangent-exact}.

Suppose $d>1$. Choose a parameter point represented by $0$ and let $e=\tau(0)$.
For small $r$, let $E_r=\Phi([r,2r])$ and $F_r=\Phi([-2r,-r])$. Then $e\cdot\tau(u)\geq1/2$ for $|u|\leq2r$.
For $t\in[r,2r]$, $s\in[-2r,-r]$,
$$
 r\leq\tfrac12(t-s)\leq e\cdot(\Phi(t)-\Phi(s))
 \leq|\Phi(t)-\Phi(s)|\leq t-s\leq4r.
$$
The tangential component of the kernel therefore satisfies
$$
 e\cdot\bfK^{\alpha,n}(\Phi(s),\Phi(t))
 \geq\frac{c_{\alpha,n}}{2(t-s)^d}\geq\frac{c_{\alpha,n}}{2(4r)^d}.
$$
Take $f_r=\one_{E_r}/\sqrt{\sigma(E_r)}$, $\varepsilon<r$ and $R>4r$. The cutoff is one on $F_r\times E_r$.
Since $\sigma(E_r),\omega(F_r)\geq mr$,
$$
 \|\one_{F_r}T_{\varepsilon,R,\sigma}f_r\|_{L^2(\omega)}
 \geq Cr^{-d}\sqrt{\sigma(E_r)\omega(F_r)}\geq C'mr^{1-d}\to\infty.
$$
Any operator agreeing with the kernel on separated supports satisfies the same lower bound and is therefore unbounded.
\end{proof}

\section*{Statements and Declarations}

\noindent\textbf{Funding.}
Ji Li is supported by the Australian Research Council through grant DP260100485.  Chun-Yen Shen is supported in part
by the National Science and Technology Council of Taiwan through grant 111-2115-M-002-010-MY5. Brett D. Wick is
partially supported by the National Science Foundation under grant DMS-2349868 and by the Australian Research
Council under grant DP260201083.

\medskip
\noindent\textbf{Data availability.}
No datasets were generated or analysed during this study.

\medskip
\noindent\textbf{Use of AI tools.}
During the preparation of the manuscript, OpenAI’s ChatGPT was used to assist with
proof auditing and language refinement. All mathematical arguments and references were
independently verified by the authors, who take full responsibility for the content of the
paper.

\end{document}